\documentclass[11pt, a4paper,reqno]{amsart}
\usepackage[utf8]{inputenc}
\usepackage{enumerate}
\usepackage{amsmath, amssymb,amsthm}
\usepackage{mathtools}
\usepackage{enumitem}
\usepackage{leftidx}
\usepackage{amsmath}
\usepackage{amssymb}
\usepackage{amsthm}
\usepackage{graphicx}
\usepackage{mathtools}
\usepackage{amsfonts}
\usepackage{amssymb}
\usepackage{hyperref}
\usepackage{bm}
\usepackage{tikz}
\usepackage{tikz-cd}
\usepackage{yfonts}
\usepackage{comment}

\numberwithin{equation}{section}
\usepackage[left=2.5cm, right=2.5cm,bottom=4cm]{geometry}
\usepackage{hyperref}
\usepackage{xcolor}
\definecolor{my-black}{rgb}{0,0,0}
\definecolor{my-blue}{rgb}{0,0,0.8}
\definecolor{my-red}{rgb}{0.8,0,0} 
\definecolor{my-green}{rgb}{0,0.5,0}
\hypersetup{colorlinks,
	linkcolor	= {my-blue},
	citecolor	= {my-red},
	urlcolor		= {my-black}
}
\usepackage{graphicx}
\usepackage{tikz-cd}

\theoremstyle{plain} 

\newtheorem{lemma}{Lemma}[section]
\newtheorem{theorem}{Theorem}[section]
\newtheorem{corollary}[theorem]{Corollary}
\newtheorem{proposition}{Proposition}[section]
\newtheorem{assumption}{Assumption}[section]

\newtheorem*{theorem*}{Theorem} 

\theoremstyle{definition} 
\newtheorem{remark}{Remark}[section]

\newtheorem{definition}{Definition}[section]

\newtheorem{example}{Example}[section]

\theoremstyle{remark}
\newtheorem*{remark-non}{Remark}

\newcommand{\R}{\mathbb{R}}

\newcommand{\eps}{\varepsilon}

\newcommand{\ba}{\begin{aligned}}
\newcommand{\ea}{\end{aligned}}
\newcommand{\be}{\begin{equation}}
\newcommand{\ee}{\end{equation}}

\title{On the doubling of variables technique in first order Hamilton-Jacobi equations}

\author{Charles Bertucci}
\author{Giacomo Ceccherini Silberstein}

\address[1]{CEREMADE, CNRS, Universit\'e Paris Dauphine-PSL, UMR 7534, 75016 Paris, France.}
\address[2]{Dipartimento di Matematica ``Tullio Levi Civita'', Università di Padova, 
Via Trieste 63, 35121, Italy}

\begin{document}

\begin{abstract}
In this paper, we revisit the technique of doubling variables in first order Hamilton-Jacobi equations, especially when the equations arise in optimal control. We show that by tuning the penalization between the two points, we can change drastically the proof, somehow shifting the regularity hypotheses into geometrical properties of the penalization. We present this idea in a finite dimensional setting and then exploit it on equations posed on Wasserstein spaces.\end{abstract}

\maketitle

\setcounter{tocdepth}{1}
\tableofcontents

\section{Introduction}

\subsection{Doubling of variables and central question of the paper}

The main motivation of this paper is to reinterpret and slightly change the usual method for proving a comparison principle on viscosity solutions of the Hamilton–Jacobi equation
\begin{equation}\label{e: HJ}
  u(x) + H(x,d_x u) = F(x), \qquad x\in M,
\end{equation}
where $M$ is a boundaryless Riemannian manifold, \(H:T^*M\to\mathbb{R}\) is the Hamiltonian, \(F:M\to\mathbb{R}\) is a given datum, and \(u:M\to\mathbb{R}\) is the unknown function.

Such a reinterpretation is itself motivated by the analysis to the infinite-dimensional, non-smooth setting of the space of probability measures endowed with the \(p\)-Wasserstein distance. In that context one considers the mean-field (MF) Hamilton–Jacobi equation
\begin{equation}\label{eq: WHJ}
  U(\mu) + \int_{M} H\big(\mu,x,d_\mu U(\mu,x)\big)\,d\mu(x) = \mathcal{F}(\mu),
  \qquad \mu\in\mathcal P_p(M),
\end{equation}
where   $p>1$, $\mathcal P_p(M)$ is the $p$-Wasserstein space over $M$, $H:P_p(M)\times T^*M \to\mathbb{R}$ and $ \mathcal F:\mathcal P_p(M)\to\mathbb{R}$ are the data and $ U:\mathcal P_p(M)\to\mathbb{R}$ is the unknown. The facts that $\mathcal P_p(M)$ is neither a vector space, nor a Riemannian manifold make the study of \eqref{eq: WHJ} particularly difficult. Furthermore, we shall come back on the motivation for such equation later on.\\

Omitting several technical aspects, we reproduce the main argument of the proof of the comparison principle in the periodic case (i.e. $M = \mathbb T^d$, the $d$ dimensional torus), to explain our strategy. Considering $u$ and $v$, respectively sub and super-solution of \eqref{e: HJ}, we take $\eps > 0$ and define\footnote{We work as in the periodic setting in what follows in this introduction.}
$$
w(x,y) = u(x) - v(y) -\frac{1}{2\eps}|x-y|^2.
$$
Considering a point of maximum $(x_\eps,y_\eps)$ of $w$, we obtain that $p_\eps:=\eps^{-1}(x_\eps - y_\eps) \in \partial^+u(x_\eps)\cap\partial^-v(y_\eps)$. Using that $u$ and $v$ are viscosity solutions, we obtain 
$$
\max(u-v) \leq u(x_\eps) - v(y_\eps) \leq H(y_\eps,p_\eps)-H(x_\eps,p_\eps) + F(y_\eps) - F(x_\eps).
$$
The proof then concludes under assumptions of the type: $F$ is continuous and there exists $C > 0$ such that for all $p \in \R^d,x,y\in \mathbb{T}^d$, $|H(x,p) - H(y,p)| \leq C(1+ |p|)|x-y|$. Indeed in this case, we arrive at 
$$
\max (u-v) \leq \frac C\eps|x_\eps -y_\eps|^2  + o(1) = o(1),
$$
since $\eps^{-1}|x_\eps -y_\eps|^2 \to_{\eps \to 0}0$ under very mild assumptions. We refer to \cite{crandall1992usersguideviscositysolutions} for much more involved developments on this question.\\

Adapting this proof to the case of $\mathcal P_p(M)$ is quite non-trivial. A first answer has been given in \cite{bertucci2024stochasticoptimaltransporthamiltonjacobibellman} in the case of $\mathcal P(\mathbb T^d)$, namely by replacing $|x-y|^2$ by an appropriate Wasserstein distance, and then using an appropriate notion of super-differentiability of such functions to create elements in the super/sub differentials of $u$ and $v$. Studying cases in which $M$ has a non-flat geometry remains open, such as many variants of \eqref{eq: WHJ} in which more singular terms appear. Furthermore, the authors are currently studying an optimal control of positive measures endowed with Hellinger-Kantorovich like distances \cite{bertucciceccherini}, in which simply using the Wasserstein distance does not suffice. Hence, we believe that a novel point of view on the proof of doubling of variables is required to solve several HJ equations on Wasserstein spaces. We do not claim to completely solve the problem with what follows, but we believe our approach can be insightful in many cases.\\

We start with the following elementary remark that, when $H(x,p) = |p|^2$, the previous proof \emph{simplifies} in the sense that we do not need to have a quantitative estimate on the error in the two terms in $H$ since they cancel each other perfectly. Of course, such a quadratic Hamiltonian satisfies the previous estimate but we argue that it is not used here. Indeed, in this quadratic case, we can interpret the penalization term $(2\eps)^{-1}|x-y|^2$ as the cost to go from $x$ to $y$ in time $\eps$ if we pay a quadratic cost $\frac12|\alpha|^2$ on the speed $\alpha$ that we choose. Moreover, in this case, the quadratic Hamiltonian is exactly the one associated to the quadratic cost.

A similar fact happens in the purely Wasserstein setting. The main objective of this paper is thus to make precise the following heuristics which is a generalization of the previous remark: when trying to prove a comparison principle for an HJB equation associated to an optimal control problem, it is quite natural to use a penalization function which is constructed with the optimal control problem itself. More generally, even when the Hamiltonian do not arise from an optimal control problem (i.e. when it is not convex), we show that a the geodesic distance is an effective penalization to use in doubling of variables, in both \eqref{e: HJ} and \eqref{eq: WHJ}.

\subsection{Structure of the paper}

In Section \ref{s:Assumption}, we list the main assumptions, recall the required notions of differential geometry and superdifferential calculus on $M$, and some properties of action functionals that will be useful later. 

 Section \ref{s: Finite Hamilton Jacobi} is devoted to defining the notion of a viscosity solutions for \eqref{e: HJ}, as well as establishing suitable comparison principles for this notion. We address both the cases in which $H$ is non-convex with the geodesic distance as a penalization and the case of certain convex $H$ with an exactly appropriate penalization. It will be apparent in the proofs that both arguments have a global character, as they both make essential use of the global super-differentiability of the penalization function (Proposition \ref{superdifferentiability}) rather than localizing everything in a single chart to use standard techniques in $\R^d$.

In Section \ref{s: PrelWass}, we introduce the mean-field Lagrangian framework together with the associated Fenchel duality.  
We also present an equivalent formulation of the optimal transport problem (see Theorem \ref{t: equivalent}), which will play a crucial role in the superdifferential calculus on $\mathcal{P}_p(M)$.

Finally, in Section \ref{s: wComparison}, we define the notion of viscosity solution to \eqref{eq: WHJ} and establish the analogues of the two results (convex or non-convex $H$) that we obtained in Section \ref{s: Finite Hamilton Jacobi}.

We remark that the time-evolutive HJ equation can also be treated with the same methods, with adaptations that are standard in the literature of viscosity solutions.

The rest of the introduction is devoted to bibliographical comments.

\subsection{Bibliographical comments}
Doubling of variables techniques for Hamilton-Jacobi equations are deeply linked with the theory of viscosity solutions, developed by Crandall and Lions \cite{crandall1983viscosity} in similar cases as the one presented above but set in $\R^d$. Quite rapidly, it has been apparent that similar techniques could be used in infinite dimensional Banach spaces \cite{CrandallLions1984}. We remark that the developments of the theory to second order equations was much more involved, as explained in \cite{crandall1992usersguideviscositysolutions}. Note that the extension of the theory of viscosity solutions to second order equations was possible because of a better understanding of the doubling of variables technique.

A theory of HJB equations set on Riemannian manifold has also been developped, and we refer to \cite{Fathi_2007} for more details on this topic.

More recently, HJ equations on Wasserstein spaces have gained a lot interest recently, for mainly two reasons. The first one is their use in the MF control and the second is the development of powerful tools of analysis on such spaces in the theory of mean field games, initiated by Lasry and Lions \cite{LLMeanFG}, tools which complemented the already existing ones, presented in details in \cite{AGS}. See \cite{bertucci2026} for an overview of analysis on spaces of measures. The first major theoretic work on such HJ equations was done by Lions and presented in \cite{College}, using a so-called lifting approach. Later on, the link of this lifting technique with a more intrinsic approach was studied in \cite{GANGBO2019119}, for convex Hamiltonians. A more general point of view was then adopted in \cite{bertucci2024stochasticoptimaltransporthamiltonjacobibellman} in a compact case, see also \cite{aussedat}. Numerous works are also concerned with more singular equations such as \cite{BertucciApprox}, modelling the presence of either common or idiosyncratic noises, using quite different techniques than the ones used here, so we do not enter in this long literature here and refer to the introduction of \cite{bertucci2024stochasticoptimaltransporthamiltonjacobibellman} for more details.

\section{Notation and assumptions}\label{s:Assumption}

\subsection{Structure of the ambient space} Let $M$ be a smooth connected manifold without boundary and  $TM, T^*M$ its tangent and cotangent bundle, respectively. We denote by
\[
\begin{aligned}
    \pi \; & : \; TM \to M, \quad \pi^{*} \;  : \; T^{*}M \to M \\
     &\quad (x,v) \mapsto x, \quad (x,p) \mapsto x
\end{aligned}
\]
the correspondent natural projections, and by $p(v)$ the evaluation of an element $p \in T^{*}M$ at $v \in TM$. 
We fix a complete Riemannian structure: We denote by $g$ the Riemannian metric on $M$, and by $g^*$ its corresponding dual. The associated norms are $\|\cdot\|$ and $\|\cdot\|_{*}$, respectively.

The integral length of a smooth curve $\gamma:[a,b]\to \R$ is defined as 
\begin{equation*}
    \text{Length}_{g}(\gamma):=\int_a^b \|\dot \gamma(t)\|_{{\gamma(t)}}dt.
\end{equation*}
 For every $I \subseteq \R$ closed interval we define 
\begin{equation}
    \Gamma_I(x,y)=\Big\{\gamma: I \to M \Big| \gamma \text{ piecewise } C^1, \gamma(\inf I)=x, \gamma(\sup I)=y \Big\},
\end{equation}
to be the our set of admissible curves.
The function $d: M \times M \to [0,\infty)$ defined by
\begin{equation*}
    d(x_0,x_1):=\inf_{\gamma \in \Gamma_{[0,1]}(x_0,x_1)} \text{Length}_{g}(\gamma)
\end{equation*}
is the geodesic distance between the two points $x,y\in M$.

Saying that $g$ is complete is equivalent to the completeness of the metric space $(M,d)$. See Hopf-Rinow's  \cite[Theorem~2.8]{doCarmo}.
Equivalently (geodesic completeness), the exponential map associated to $g$, $\text{exp}: TM \to M$, is well defined on the entire tangent bundle.
Fix $x,y \in M$, we denote by $\mathcal{T}_y^x[\gamma]: T_yM \to T_xM$ the parallel transport along the curve $\gamma : [0,1] \to M$ connecting $x=\gamma(0)$ and $y=\gamma(1)$. The parallel transport along any curve is an isometry, see  \cite[p. ~56]{doCarmo}. Furthermore, parallel transport along geodesic has the following property
\begin{equation}\label{PTgeodesic}
\mathcal{T}_{\gamma(s)}^{\gamma(t)}[\gamma](\dot \gamma(s))=\dot \gamma(t), \quad \forall s\leq t 
\end{equation}
Given $x,y \in M$,
we will denote by $\text{Geo}_M(x,y)$ the set of \emph{minimizing} geodesics connecting $x$ and $y$.\\

We finally, we recall the definition of the Sasaki distance $D_S$ on the tangent bundle $TM$. It is defined as $D_S: TM \times TM \to [0,\infty)$:
\[
D_S((x,v), (y,z)) = \inf_{\gamma \in \Gamma_{[0,1]}(x,y)} 
\left\{ 
\int_0^1 \|\dot \gamma(s)\|^2_{\gamma(s)}\,ds 
+ \|\mathcal T_y^x[\gamma](z)-v\|^2_x 
\right\}^{1/2},
\quad (x,v),(y,z) \in TM.
\]
See also \cite[p.~10]{Gigli}. This distance satisfies the following properties:
\begin{enumerate}
\item Behavior along vertical and horizontal directions 
\begin{align}\label{e: vertical}
    D_S\big((x,v), (x,v+w)\big) &= \|w\|_x 
    \qquad \forall (x,v) \in TM, \; w \in T_xM.
\end{align}
\begin{align}\label{e: horizontal}
    d(x,y)&\leq D_S\big((x,v), (y,w)\big)
    \qquad \forall (x,v),(y,w) \in TM,
\end{align}
and equality holds if $v=0, w=0$, and whenever $v=\dot\gamma(0)$, $w=\dot\gamma(1)$, where $\gamma:[0,1]\to M$ is a minimizing geodesic connecting $x,y$. 
\item In particular, by the triangle inequality,
\begin{equation}
D_S\big((x,v+w), (y,v'+w')\big)
    \leq \|w\|_x + \|w'\|_y + D_S\big((x,v), (y,v')\big)
    \qquad \forall (x,v),(y,v') \in TM, \; w \in T_xM, \; w' \in T_yM.
\end{equation}
\item For every minimizing geodesic $\gamma: [0,1] \to M$ connecting $x,y \in M$, the following bound holds:
\begin{equation}
    D_S\big((x,v),(y,w)\big)
    \leq d(x,y) + \|w - \mathcal{T}_x^y[\gamma](v)\|_y.
\end{equation}
See also \cite[p. ~10]{Gigli}.
\item $\forall (x,v),(y,w)\in TM$ we have 
\begin{equation}\label{e: boundnormedistanza}
    \|v\|_x-\|w\|_y\leq 2D_S((x,v),(y,w)) \quad \forall (x,v),(y,w)\in TM. 
\end{equation}
In fact, consecutive application of the triangular inequality shows that for all $(x,v),(y,w)\in TM$ we have
\begin{align*}
     \|v\|_x-\|w\|_y&\underbrace{=}_{\eqref{e: vertical}} D_S((x,0),(x,v))-D_S((y,0),(y,w))\\
     &\leq D_S((x,v),(y,w))+D_S((x,0),(y,w))-D_S((y,0),(y,w))\\
     &\leq D_S((x,v),(y,w)) + D_S((x,0),(y,0))\\
     &\underbrace{=}_{\eqref{e: horizontal}}  D_S((x,v),(y,w)) + d(x,y) \\
     &\underbrace{\leq}_{\eqref{e: horizontal}} 2 D_S((x,v),(y,w)).
\end{align*}
\end{enumerate}

\subsection{Assumptions on the Hamiltonian and the associated Lagrangian in the convex setting}\label{sec:hypgeo}
As stated in the introduction, we shall study the equations \eqref{e: HJ} and \eqref{eq: WHJ} in two regimes: a general one under a classical regularity assumption on the Hamiltonian, and one where the Hamiltonian is convex and where we want to use more geometrical arguments. In this second case, we restrict ourselves to Hamiltonians satisfying the following requirements, which we will call \emph{geometric Hamiltonians} for the sake of convenience. The following requirements are taken from \cite{Fathi_2007} and \cite{Fathi2010}.

A geometric Hamiltonian is a function $H:T^*M \to \R$ of class $C^1$ which satisfies
\begin{enumerate}[label=(H\arabic*), ref=H\arabic*, leftmargin=2em]
    
    \item \label{hypoH1} For all $x \in M$, the map $z \mapsto H(x,z)$ is strictly convex and superlinear on $T^{*}_xM$.
    
    \item \label{hypoH2}\textit{Uniform superlinearity property:} For every $K \geq 0$, there exists a constant $C^*(K) \in \mathbb{R}$ such that
    \[
    \forall (x,z) \in T^*M, \quad H(x,z) \geq K {\|z\|_{*}}_x - C^*(K).
    \]

    \item \label{hypoH3}\textit{Uniform boundedness property:} For every $R \geq 0$, we have
    \[
    A^*(R) = \sup \left\{ H(x,z) \;\middle|\; {\|z\|_{*}}_x  \leq R, \, x \in M \right\} < \infty.
    \]
\end{enumerate}
Associated to a geometric Hamiltonian, we consider its \textit{Lagrangian} $L: TM  \to \R$ defined via the Fenchel Transform:
\begin{equation}
    L(x,v):=\max_{z \in T_x^{*}M} z(v)-H(x,z).
\end{equation}

It is easy to show that the previous assumptions on $H$ imply that $L$ satisfies the following properties
\begin{enumerate}[label=(L\arabic*), ref=L\arabic*, leftmargin=2em]
    \item \label{hypoL1} It is finite and of class $C^1$ on $TM$.
    
    \item \label{hypoL2} For all $x \in M$, the map $v \mapsto L(x,v)$ is strictly convex and superlinear on $T_xM$.
    
    \item \label{hypoL3} $L$ is dual to $H$, i.e., the following duality holds:
    \begin{equation}
        H(x,p) = \max_{v \in T_xM} \left[ p(v) - L(x,v) \right].
    \end{equation}
\end{enumerate}
The \emph{Legendre transform} is the homeomorphism (Proposition B.9 \cite{Fathi2010}) $\mathcal{L}
: TM \to T^*M$ defined by 
\begin{equation} \label{legendre}
\mathcal{L}(x,v)=(x,\frac{\partial L}{\partial v}(x,v)).
\end{equation}
Moreover, we have for all $(x,z,v)$ such that $(x,z) \in T^{*}M,  (x,v) \in TM$,
\begin{equation}\label{e: legendreequality}
    z(v)=H(x,z)+L(x,v) \iff (x,z)=\mathcal{L}(x,v).  
\end{equation}

Thus, $\mathcal{L}$ can be seen as a (nonlinear) duality map between $TM$ and $T^{*}M$.
In the case $L=\frac{g^2}{2}$, the Legendre transform coincides with the usual (linear in $T_xM$) duality map induced by the Riemannian structure, i.e. $\mathcal{L}(x,v)=g_x(v,\cdot)$. For convenience in this setting (See also Example \ref{ex: p-norm}), we set $J_2(x,v):=\mathcal{L}(x,v)$. We remark for future use that 
\begin{equation}\label{e: isometrydual}
    g^{*}_x(J_2(x,v), J_2(x,w))=g_x(v,w), \quad \forall v,w \in T_xM.
\end{equation}

Under the assumptions on $H$ it can be proved (\cite[Lem. ~2.1]{Fathi_2007}) that the following two properties hold

\begin{enumerate}[label=(L\arabic*), ref=L\arabic*, leftmargin=2em, start=4]
    \item \label{hypoL4} \textit{Uniform superlinearity property:} For every $K \geq 0$, there exists $C(K) \in \mathbb{R}$ such that
    \begin{equation*} \label{growthL}
        \forall (x,v) \in TM, \quad L(x,v) \geq K \|v\|_x - C(K).
    \end{equation*}

    \item \label{hypoL5} \textit{Uniform boundedness property:} For every $R \geq 0$, we have
    \begin{equation*}
        A(R) = \sup \{ L(x,v) \mid \|v\|_{x} \leq R, \,  x \in M \} < \infty
    \end{equation*}
\end{enumerate}
In the literature, under our assumptions, $L$ is said to be a \emph{weak Lagrangian} (\cite[Def. ~ B.4]{Fathi2010}).\footnote{There is, however, a difference for \eqref{hypoL4}--\eqref{hypoL5}.  
In \cite{Fathi2010}, only \emph{local} superlinearity and boundedness are assumed,  
whereas here we adopt stronger \emph{uniform} conditions.  
This choice is motivated by their applicability in the Wasserstein setting (see the last sections).
}
\begin{definition} \label{d: dissipativeLagrangian}
Let $L$ be a weak Lagrangian.
Whenever $L\geq 0$ and $L(x,0)=0, \, \forall x \in M$, we say that $L$ is a \textit{dissipative Lagrangian}.
A weak Lagrangian is said to be a \emph{Tonelli Lagrangian} if $L$ is $C^2$ and strictly convex in each fiber, in the $C^2$ sense; that is, the second vertical derivative
$D^2_v L(x,v)$
is positive definite, as a quadratic form, for all $(x,v) \in TM$.

\end{definition}

The \textit{Dual Energy}  $\hat H \colon  TM \to \R$ is defined as
\begin{equation}\label{Dual Hamiltonian}
    \hat H(x,v):=H(x, \frac{\partial L}{\partial v}(x,v)).
\end{equation}

\begin{example}\label{ex: p-norm}
Consider $L(x,v)= \frac{\|v\|_x^{p}}{p}$, where $p> 1$.  Let  $q$ be the conjugate exponent to $p$, i.e. $\frac{1}{p}+\frac{1}{q}=1$.
Then $L$ is a $C^1$ and weak Tonelli Lagrangian (See also Example B.5 \cite{Fathi2010}). Moreover, it is dissipative. In this case, 
$$J_{p}(x,v):=\frac{\partial L}{\partial v}(x,v)=\|v\|^{p-2}_x g_x(v,\cdot)=\|v\|^{p-2}_xJ_2(x,v), \, \forall (x,v)\in TM.$$ 
We also note for future use that 
\begin{equation}\label{dualitynorms}
    {\|J_{p}(x,v)\|_{*}}_x=\|v\|^{p-2}_x {\|J_{2}(x,v)\|_{*}}_x= \|v\|^{p-1}_x, \quad  \forall (x,v)\in T^*M. 
\end{equation}
In particular,
$$
{\|J_{p}(x,v)\|^{q}_{*}}_x=\|v\|^{(p-1)q}_x=\|v\|^{p}_x
$$
The associated Hamiltonian is $H(x,z)=\frac{{{\|z\|}^q_{*}}_{x}}{q}$. In fact, if $z=J_p(x,v)$ by \eqref{e: legendreequality} 
$$
H(x,z)=J_p(x,v)(v)-\frac{1}{p}\|v\|^{p}_x= \|v\|^{p}-\frac{1}{p}\|v\|^{p}_x=\frac{1}{q}\|v\|^p=\frac{1}{q}{\|J_p(x,v)\|_{*}}^{q}_x=\frac{1}{q}{\|z\|_{*}}^{q}_x.
$$

\end{example}

We conclude this subsection recalling the reversible case, in which the Lagrangian satisfies
\begin{equation}
     L(x,v)=L(x,-v) \quad \forall (x,v)\in TM.
\end{equation}
and $L(x,0)= 0$.
In particular, by the strict convexity, fixed $x\in M$ the function $v \in T_xM \mapsto L(x,v)$ has a unique minimum at $0$, and we infer $L\geq 0$. In other words, every reversible Lagrangian s.t. $L(\cdot,0)\equiv0$  is dissipative. 
Note also that if $L$ is reversible, its associated Hamiltonian $H$ satisfies 
\begin{equation}\label{Symmetry Hamiltonian}
H(x,p)=\inf_{v \in T_{x}M}p(v)-L(x,v)=\inf_{v \in T_{x}M}-p(v)-L(x,v)=H(x,-p), \quad \forall (x,p) \in T_x^{*}M.
\end{equation}

\subsection{Action Functional} 
From now on, we suppose the Lagrangian $L$ to satisfy \eqref{hypoL1}--\eqref{hypoL5}.

Given such $L$ and $I \subseteq \R$ a closed interval, we define the associated action as 
\begin{equation*}
    \mathcal{A}_{I}(\gamma):=\int_{I}L(\gamma(t),\dot \gamma(t))dt.
\end{equation*}
Fix $t>0$, the \textit{minimal action} is defined as 
\begin{equation}\label{action}
    D(t,x,y):=\inf_{\gamma \in \Gamma_{[0,t](x,y)}}\mathcal{A}_{[0,t]}(\gamma).
\end{equation}

\begin{example}[Minima of the action functional and geodesics]
\label{Lagrangian&Geo}

Given $p> 1$, we consider the Lagrangian $L(x,v)=\frac{g_x(v,v)^{p}}{p}$. By H\"older-inequality, for every $\gamma \in \Gamma_{I}(x_0,x_1)$
\begin{equation*}
    \text{Length}_g(\gamma)^p\leq |I|^\frac{p}{q}\mathcal{A}_{I}(\gamma), 
\end{equation*}
where $q$ is the conjugate exponent to $p$, and $|I|$ is the lenght of the interval $I$. The equality holds iff $g_{\gamma(t)} (\dot \gamma(t),\dot \gamma(t))$ is a.e. constant, i.e. the parametrization is proportional to the arc lenght.
Since minimizing geodesic are parametrize by the arc lenght, we have that, if $\gamma$ is a minimizing geodesic, then
\begin{equation*}
    \vert I \vert^{\frac{p}{q}} \mathcal{A}_{I}(\gamma)= \text{Length}_g(\gamma)^{p}\leq \text{Length}_g(c)^{p}\leq \vert I \vert^{\frac{p}{q}} \mathcal{A}_{I}(c), \quad  \forall c :I \to M \text{ piecewise } C^1,
\end{equation*}
with equality iff $c$ is a minimizing geodesics. Thus, the minimization problems for $\mathcal{A}_{I}$ and for $\text{Length}_g$ are equivalent, and the minima are geodesics.
Moreover, in this case we have
\begin{equation}
D(\varepsilon,x,y)=\frac{d^p(x,y)}{p\varepsilon^{p-1}}.
\end{equation}
\end{example}

More generally we have the following existence result
\begin{theorem}\label{t: existence}(\cite[Thm ~ B.6]{Fathi2010})
Suppose $L$ is a weak Lagrangian. Then for every $a,b \in \R$, $a<b$ and every $x,y \in M$, there exists an absolutely continuous curve $\gamma : [a,b] \to M$ which is a minimizer of $\mathcal{A}_{[a,b]}$ whith $\gamma(a)=x$ and $\gamma(b)=y$.
\end{theorem}

The following lemma will be useful in the comparison principle.
\begin{lemma}\label{Decreasing}
Let $L$ be a dissipative Lagrangian. 
The function $D(\varepsilon, \cdot,\cdot): M \times M \to [0,\infty)$ 
\begin{equation}
    D(\varepsilon,x,y):=\inf_{\Gamma_{[0,\varepsilon]}(x_0,x_1)} \int_0^\varepsilon L(\gamma(t), \dot \gamma(t))dt
\end{equation}
is non-negative and 
\begin{equation}
    D(\varepsilon,x,y)> D(\tau\varepsilon,x,y) \quad \forall \tau> 1.
\end{equation}

\end{lemma}

\begin{proof}
Fix $\varepsilon, \tau>0$, we denote by 
\begin{align*}
    h_\tau: &\Gamma_{[0,\varepsilon]}\to \Gamma_{[0,\tau\varepsilon]}\\
    & \gamma(t) \mapsto \gamma(\frac{t}{\tau})
\end{align*}
Obviously, $\frac{d}{dt}h_\frac{1}{\tau}(\gamma)(t)=\frac{1}{\tau}\dot\gamma(\frac{t}{\tau})$ $t-\, a.e$. 
Let $\gamma \in \Gamma_{[0,\varepsilon]}$ be an optimizer for the action functional (\ref{t: existence}), and fix $\tau>1$. Then we have 
\begin{align*}
D(\varepsilon,x,y) &= \int_0^{\varepsilon} L(\gamma(t), \dot{\gamma}(t)) \, dt \\
&= \frac{1}{\tau} \int_0^{\tau \varepsilon} L(h_{\frac{1}{\tau}}(\gamma), \dot\gamma(\frac{t}{\tau})) \, dt\\
&\underbrace{>}_{\text{strict convexity} \& L(\cdot,0)=0}  \int_0^{\tau \varepsilon} L(h_{\frac{1}{\gamma}}(\gamma), \frac{1}{\tau}\dot \gamma(\frac{t}{\tau})) \\
&= \int_0^{\tau \varepsilon} L(h_\frac{1}{\tau}(\gamma)(t), \frac{d}{dt}h_{\frac{1}{\tau}}(\gamma)(t)) \, dt \\
& \geq D(\tau \varepsilon, x,y).
\end{align*}
\end{proof}
This function $D$ is the one we shall use as a penalization in the convex case.

We now state without proof the regularity result for minimizers. 
\begin{theorem}\cite[Thm~ B.7\& Cor. ~ B.15]{Fathi2010}\label{t:EL}
If $L$ is a weak Tonelli Lagrangian, then every minimizer $\gamma: [a,b] \to M$ is $C^1$. Moreover, on every interval $[t_0,t_1]$ s.t. $\gamma([t_0,t_1])$ is contained to a chart, it satisfies the following equality written in the coordinate system
\begin{equation*}
    \frac{\partial L}{\partial v}(\gamma(t_1), \dot \gamma(t_1)) - \frac{\partial L}{\partial v}(\gamma(t_0), \dot \gamma(t_0))= \int_{t_0}^{t_1}\frac{\partial L}{\partial x}(\gamma(s), \dot \gamma(s))ds.
\end{equation*}
In particular, $\frac{\partial L}{\partial v}(\gamma(t), \dot \gamma(t))$ is $C^1$ as a function of $t$ and satisfies the Euler-Lagrange equation 
\begin{equation*}
    \frac{d}{dt}\frac{\partial L}{\partial v}(\gamma(t), \dot \gamma(t))=\frac{\partial L}{\partial x}(\gamma(t), \dot \gamma(t)).
\end{equation*}
Moreover, the energy $\hat H$ is constant on the speed curve
$s \mapsto (\gamma(s),\dot \gamma(s))$.
In addition, if $L$ is $C^r$ Tonelli Lagrangian, with $r\geq 2$, then any minimizer is of class $C^r$.
\end{theorem}

The Euler-Lagrangian (E-L) flow associated to a Tonelli Lagrangian $L$ is the flow map defined by initial data, i.e. $\phi^t(x,v)=(\gamma(t),\dot\gamma(t))$, where $\gamma$ is defined in the previous theorem, s.t. $(\gamma(0),\dot \gamma(0))=(x,v)$. Combining the previous two results we have the following, see  \cite[Cor. ~2.2]{Fathi_2007}
\begin{corollary}\label{compltness flow}
The Euler-Lagrange flow $\phi^t: TM \to TM$ of a Tonelli Lagrangian $L$ is complete (i.e. global existence in time).  
\end{corollary}

In the case of a Tonelli Lagrangian, we can then define the Lagrangian exponential map 
\begin{equation}\label{eq:2.16}
\exp^{L}_x(\varepsilon;v)=\pi(\phi^{\varepsilon}(x,v)).
\end{equation}
Seen as a map $\exp^L_x(\varepsilon;\cdot):T_xM \to M$, it is surjective since the E-L flow is complete.

\subsection{Superdifferential calculus}
We now introduce the notion of super-differentiability that we shall use when defining viscosity solutions.

\begin{definition}
Fix $x\in M$ and let $f: M \to \R$ a be an upper semi-continuous function. The \emph{super-differential} of $f$ at $x$, denoted $\partial^{+} f(x)$, is the set of $p \in T_x^{*}M$ s.t. 
\begin{equation}\label{superdifferentiability}
 f(\exp_x(v)) -f(x)\leq p(v)+ o(d(x,\exp_x(v))),    
\end{equation}
holds for all $v\in T_xM$, where $\frac{o(\lambda )}{\lambda}$ tends to zero with $\lambda \to 0$ and depends only on $x$.
The \emph{sub-differential} of $f$ at $x$, denoted $\partial^{-} f(x)$, is the set of elements $p \in -\partial^{+}(-f)(x)$. 

\end{definition}
Observe that $p\in \partial^{+}f(x)\cap \partial^{-}f(x)$ iff $f$ is differentiable in the usual sense and, $p=d_xf$.\\

We also remark that if $f,g: M \to \R$, are two functions s.t. $\partial^{+}f(x),\partial^{+}g(x)$ are both not empty for some $x \in M$, then $\partial^{+}f(x)+ \partial^+g(x)\subseteq \partial^{+}(f+g)(x) $.

\begin{example}[Super-differential of a semiconcave function]
An important class of super-differentiable functions is made of $(\lambda,\omega)$- geodesically semiconcave functions, i.e. $f: M \to \R$ s.t. there exist $\lambda: M \times M \to \R$ continuous and $\omega: [0,\infty) \to [0,\infty)$, a modulus of continuity with the following property: 
\begin{equation*}
    f(\gamma(t))\geq tf(\gamma(1))+ (1-t)f(\gamma(0))+ t(1-t)\lambda(\gamma(0),\gamma(1))\omega(d(\gamma(0),\gamma(1))),
\end{equation*}
for all $\gamma: [0,1]\to M$ geodesic.
Equivalently, by the completeness of the metric structure, we can reformulate the previous condition in terms of the exponential map: for all $(x,v) \in TM$ we have
\begin{equation*}
    f(\text{exp}_x(tv))\geq tf(\text{exp}_x(v))+ (1-t)f(x)+ t(1-t)\lambda(x,\exp_x(v))\omega(d(x,\exp_x(v)).
\end{equation*}
Observe that for $(\lambda,\omega)$-semiconcave functions we have
\begin{equation*}\
 f(\exp_x(v)) -f(x)\leq p(v)+  \lambda(x,\exp_x(v))\omega(d(x,\exp_x(v)), \quad \forall v\in T_xM, \, \forall p \in \partial^{+}f(x).    
\end{equation*}
This is a consequence of the monotonicity of the different quotient: Fix $x\in M$, $v \in T_xM$, and $p \in \partial^{+}f(x)$. Then
\begin{align*}
   f(\exp_x(v)) -f(x) &\leq  \frac{f(\exp_x(tv)) -f(x)}{t}+ (1-t)\lambda(d(x,\exp_x(v))\omega(d(x,\exp_x(v))\\
   &\leq p(v)+(1-t)\lambda(x,\exp_x(v))\omega(d(x,\exp_x(v))+ o(t).
\end{align*}
Upon sending $t \to 0$ we get the claim. 
This class will be important later: it exhibits a uniform error in the super-diferentiability condition that is simpler to integrate in the non-compact setting. 
\end{example}
\begin{example}\label{ex: semiconcave}
We list some classical examples of $(\lambda,\omega)$- geodesically semiconcave functions.
\begin{enumerate}
    \item $M=\R^d$ and $g_x(\cdot,\cdot)=\big<\cdot\,,\cdot\big>$, the Euclidean scalar product. Then we have the following alternative for the cost $D(1,x,y)=\frac{|y-x|^{p}}{p}$(see \cite[Lem.~ 10.2.1]{AGS})
    \begin{enumerate}
        \item If $p\geq 2$, $\omega(s)=\frac{s^2}{2}$ and $\lambda(x,y)=(p-1)\max\{|x|,|y|\}^{p-2}$
        \item If $p\leq 2$,
        $\omega(s)=\frac{s^p}{p}$ and $\lambda(x,y)=\frac{p2^{2-p}}{p-1}$.
    \end{enumerate}
    \item If $L:TM \to \R$ is a weak Tonelli Lagrangian on a  compact manifold $M$, then $x\mapsto D(\varepsilon,x,x_0)$ is semiconcave for fixed $\varepsilon$. In addition, if $L$ is locally Lipschitz (for instance a Tonelli Lagrangian) then we can take $\omega(s)=s^2$ and $\lambda$ a constant function. See  \cite[Thm ~B.19]{Fathi2010}  
    \item  If $(M,g)$ has nonnegative sectional curvature, then $x \mapsto \frac{d^2(x,x_0)}{2}$, for a fixed $x_0 \in M$ is semiconcave with modulus $\omega(r)=\frac{r^2}{2}$ and $\lambda\equiv 1.$ See also  \cite[Ex.~ 10.22]{villani2008optimal} and later discussion.
\end{enumerate}
\end{example}

To lighten some notation, we also introduce the \textit{Lagrangian supergradient} $\partial^{L,+}f$ is defined via duality:
$$
\text{Graph}_{\partial^{L,+}f}=\mathcal{L}^{-1}(\text{Graph}_{\partial^{+}f}). 
$$

We now state the following result of super-differentiability of our penalization function.

\begin{proposition}\label{Differentiability Penalization}
Fix $\varepsilon>0$ and 
let $\gamma_{x \to y} \in \Gamma_{[0,\varepsilon]}(x,y)$ be a minimizer of \eqref{action}. Then 
\begin{equation}
    -\frac{\partial}{\partial v}L(x,\dot\gamma _{x \to y}(0)) \in \partial^{+}D(\varepsilon,\cdot,y)(x).
\end{equation}
In particular,
\begin{equation*}
    -\dot\gamma_{x \to y}(0)\in \partial^{L,+}_{x}D(\varepsilon,\cdot,y),
\end{equation*}
whenever $L$ is reversible.
\end{proposition}

\begin{proof}
As observed in Remark \ref{Intrinsic=Existrnsic}, the intrinsic and extrinsic approaches are equivalent.  
Hence, we refer to the proof based on the intrinsic approach in \cite[Cor.~ B.20]{Fathi2010} to establish the first assertion.  
The second assertion follows directly from the definition of a Lagrangian supergradient.
\end{proof}

 We conclude this section with the following result, which replaces the proof of the useful Lemma 3.1 in \cite{crandall1992usersguideviscositysolutions}, to obtain that our penalization shall indeed vanish when taking the correct limit.

\begin{lemma}\label{Convergence Maxim}
We have the following

\begin{enumerate}
\item The (decreasing) sequence of continuous functions $D(\varepsilon, \cdot,\cdot)$ is Gamma converging to the convex indicator function over the diagonal $\Delta \subset M \times M$.

\item Given $F: M \times M \to \R$  u.s.c. with compact superlevels

$$M_\varepsilon:=\max_{x,y \in M\times M} \Big\{F(x,y)-D(\varepsilon,x,y)\Big\} \  \quad\uparrow_{ \varepsilon \downarrow 0} \quad M_0=\max_{x\in M} F(x,x),$$

and there exists a subsequence of $(x_\varepsilon,y_\varepsilon)\in \mathrm{argmax}_{(x,y)\in M}\big\{F(x,y)-D(\varepsilon,x,y)\big\}$ that converges to a point of maximum for $M_0$.

\end{enumerate}
\end{lemma}

\begin{proof}
Fix $t> 0$ and $x,y \in M$ .
By\eqref{hypoL4}, we have, $\forall K\geq 0$
\begin{align*}
    D(t,x,y)&=\inf_{\gamma \in \Gamma_{[0,t]}} \int_0^tL(\gamma(s),\dot \gamma(s))ds\\
    &\geq K \inf_{\gamma \in {\Gamma_{[0,t]}}} \int_0^t \|\dot \gamma(s)\|_{\gamma(s)}ds - t C(K).
\end{align*}
Then we choose $\gamma :[0,t] \to M$ geodesic connecting $x$ and $y$ in time $t$. With this choice of $\gamma$,  $ \|\dot \gamma(s)\|_{\gamma(s)}=\frac{d(x,y)}{t}$. 

Therefore
\begin{equation*}
    C(K)t+ D(t,x,y)\geq K d(x,y).
\end{equation*}
The arbitrariness of $K$ implies that the sequence $D_t$ is pointwise converging to $1_{\Delta}.$ 
Due to the monotonicity stated in Proposition \ref{Decreasing}, the family of continuous functions $D_t$
is increasing as $t \downarrow 0$. By \cite[Rmk~2.12]{bra06}, the $\text{Gamma}$-limit of $(D_t)_{t>0}$ as $t \to 0$ coincides with the lower semicontinuous envelope of the pointwise limit, namely $1_\Delta$, which is already lower semicontinuous. This proves the first point. 

In particular, it follows that the functionals 
\[
F_t := -F +D_t
\longmapsto_{t \to 0} 
F_0 := -F +1_\Delta,
\] 
in the sense of Gamma convergence.

By the Fundamental Theorem of Gamma convergence \cite[~Thm. 2.10]{bra06}, $M_t \to M_0$, and every sequence of minimizers $(x_t,y_t) \in \operatorname{argmin}_{(x,y)\in M\times M} F_t(x,y)$ admits accumulation points belonging to 
$$\operatorname{argmin}_{(x,y)\in M\times M} F_0(x,y) =\operatorname{argmin}_{z \in M} -F(z,z)
.$$
\end{proof}

\section{The case of equations in Riemannian manifold}\label{s: Finite Hamilton Jacobi}

In this section, we consider the Hamilton–Jacobi equation
\begin{equation} \label{eq:HJB1} 
    u(x)+H(x,d_xu)=F(x), \quad x \in M,
\end{equation}
where $H :T^*M \to \R$ is the Hamiltonian and $F: M \to \R$.  

The purpose of this section is to give two proofs of comparison of sub/super solutions to \eqref{eq:HJB1}. One under the mild assumption that $H$ is locally Lipschitz, and one in a more geometric framework, namely when $H$ is what we called a geometric Hamiltonian. We start by recalling the notion of viscosity sub/super-solution, and then prove successively the two results.

\subsection{Viscosity solutions}
We use the following. 
\begin{definition}
We say that an upper-semicontinuous function $u : M \rightarrow \R$ is a \textit{viscosity sub-solution} of \eqref{eq:HJB1} 
	if, for all $x \in M$ and $p \in \partial^+ u(x)$, we have
\begin{align*}
	u(x)+{H(x,p)} \leq F(x).
\end{align*}
We say that $u :M \rightarrow \R$ lower-semicontinuous is a \textit{viscosity super-solution} of \eqref{eq:HJB1} if, for all $x \in  M$ and $p \in \partial^- u(x)$, we have
\begin{align*}
	u(x) + {H(x,p)} \geq F(x).
\end{align*} We say that $u$ is a \textit{viscosity solution} if it is both a sub- and super- solution.
\end{definition}

\begin{remark}\label{Intrinsic=Existrnsic}
The previous definition is equivalent to the intrinsic (test function) formulation, see Definition A.2 in \cite{Fathi2010}. 
Indeed, given any \(p \in \partial^+_x f\), there exists a function \(\phi \in C^1_c(M)\) such that \(p = d_x \phi\) and \(f - \phi\) attains a local maximum at \(x\). To construct such a \(\phi\), let \(U \subset M\) be a coordinate chart compatible with \(\exp_x\), i.e. $\exp^{-1}_{x}: U \to T_xM$ is a diffeomorphism with the image. Then, by Proposition 3.17 in \cite{cannarsa.sinestrari:04:semiconcave}, we can find a \(C^1\) function \(\tilde \phi: \exp^{-1}_x(U) \to \mathbb{R}\) such that \(d_o \tilde \phi =  p\) and \(f\circ \exp_{x}^{-1} - \tilde \phi\) has a local maximum at the origin. Therefore, we set $\phi =\tilde \phi\circ \exp^{-1}_x: U \to \R$ and observe $d_{x}\phi=d_{o}\exp_x^{-1}d_{o}\tilde{\phi}=d_{o}\tilde \phi=p$. Then, we multiply \(\phi\) by a smooth cutoff to obtain a compactly supported \(C^1\) extension of \(\phi\) to all of \(M\). 

\end{remark}
\subsection{Result for non-convex Hamiltonians}

The following assumption is a local Lipschitz continuity assumption for the Hamiltonian $H \circ J_p: TM \to \R$, measured with respect to the Sasaki distance.

\begin{assumption}\label{assumption: Hamiltonian}
For some $p>1$ we have that
\begin{equation}\label{continuityHamiltonian}
   |{H}(x,J_{p}(x,v)) - {H}(y,J_{p}(y,w))| \leq C\big(1 + \|v\|^{p-1}_x + \|w\|^{p-1}_y \big) D_S((x,v), (y,w))
\end{equation}
holds $ \forall (x,v),(y,w) \in TM.$

If either the manifold $M$ is not compact or $p\not=2$ we also prescribe that:
 \begin{equation*}
   |{H}(x,J_{p}(x,v)+J_p(x,w)) - {H}(x,J_{p}(x,w))| \leq C (1+ \|w\|_x+ \|v\|_{x})\|v\|^{p-1}_x \quad \forall (x,v),(x,w)\in TM.  
 \end{equation*}

\end{assumption}

\begin{example}[Mechanical Hamiltonian]
Fix $q>1$. The Hamiltonian $H(x,z)=\frac{\|z\|^{q}_{x*}}{q}$ satisfies Assumption \ref{assumption: Hamiltonian} with $p=q*$. Indeed, we have $H(x,J_{p}(x,v))=\frac{1}{q}\|v\|^p_x$, and 
\begin{enumerate}
    \item for every $(x,v),(y,w)\in TM$
    \begin{align}H(x,J_{p}(x,v))-H(y,J_{p}(y,w))&=\frac{1}{q}(\|v\|^p_x-\|w\|^{p}_y)=\frac{1}{q}(\|v\|^{p-1}_x+\|w\|^{p-1}_y)(\|v\|_x-\|w\|_y)\\
    &\leq \frac{2}{q}(\|v\|^{p-1}_x+\|w\|^{p-1}_y)D_S((x,v),(y,w)),
    \end{align}
    \item for every $(x,v),(x,w)\in TM$ 
    \begin{align*}
    H(x,J_{p}(x,v)+J_{p}(x,w))-H(x,J_{p}(x,v))\leq &\frac{1}{q}(\|J_{p}(x,v)+J_{p}(x,w))\|^{q-1}_{*x}+\|J_{p}(x,v)\|^{q-1}_{*x})\\
    &(\|J_{p}(x,v)+J_{p}(x,w))\|_{*x}-\|J_{p}(x,v)\|_{y*})\\  
    &\leq \frac{1}{q}(\|J_{p}(x,v)\|^{q-1}_{*x}+\|J_{p}(x,v)\|^{q-1}_{*x})\|J_{p}(x,w))\|_{*x}\\
    &=\frac{1}{q} (\|v\|_{x}+\|w\|_{x})\|w\|_{x}^{p-1}.
    \end{align*}
\end{enumerate}
\end{example}

\begin{remark}
We observe that when $M = \R^d$, $g=g_{\text{Eucl}}$, and $p = 2$, the map $J_2$ reduces to the identity, and the parallel transport acts trivially as the identity as well. In this case, the condition simply recovers the classical Lipschitz regularity required for the comparison principle; see \cite{crandall1992usersguideviscositysolutions}. Moreover we stress the fact that for $p=2$ the duality map is linear and in this case $J_2(x,v+w)=J_2(x,v)+J_2(x,w)$ and therefore the condition \eqref{continuityHamiltonian} is sufficient to describe both properties required in Assumption \eqref{assumption: Hamiltonian}.
\end{remark}

The following comparison result holds.
\begin{theorem}\label{p: comparisongeneral}
Let $H:T^*M \to \R$ which satisfies Assumption \ref{assumption: Hamiltonian} for some fixed $p>1$. 
Let $F_0,F_1:M \to \R$ be two proper functions such that: $F_0$ is upper semicontinuous, $F_1$ is lower semicontinuous.
Let $u_0,u_1: M \to \R$ be, respectively, a sub-solution, with at most $p$-growth, and a super-solution, with at most $p$-growth,  of
\[
w(x) + H(x,d_x w) = F_i(x), \qquad x\in M.
\]
Then
\[
\sup u_0-u_1 \leq\;\sup F_0-F_1.
\]
\end{theorem}

\begin{proof}
For the sake of presentation we concentrate the proof in the compact case. See the Remark \ref{general growth} on how treat the general situation.
Fix $x_0 \in M$, and for $\varepsilon,\delta>0$ set 
$$\Phi^{\varepsilon}(x,y):=u_0(x)-u_1(y)-\frac{d^p(x,y)}{p\varepsilon^{p-1}}$$
Let $(x_\varepsilon,y_\varepsilon) \subset M\times M$ be a sequence of maximum points for $\Phi^{\varepsilon}$. Such a sequence exists thanks to the boundedness assumptions on the functions, the u.s.c. of $\Phi^{\varepsilon}$ and the compactness of $M$. Then, one can show (This is a standard procedure, see \cite[Prop.~ 3.7]{crandall1992usersguideviscositysolutions}) that
\begin{itemize}
    \item The penalization $\frac{d^p(x_\varepsilon,y_\varepsilon)}{p\varepsilon^{p-1}}$ vanishes as $\varepsilon\to 0$,
     \item The sequence $(x_\varepsilon,y_\varepsilon)$ is s.t. $\Phi^{\varepsilon}(x_\varepsilon,y_{\varepsilon}) \to \max_{M}\big\{u_0-u_1\big\},$ and $(x_{\varepsilon},y_{\varepsilon})\to (\bar x,\bar x) \in M \times M$ maximum point of $u_0-u_1$.
\end{itemize}
Moreover, by the maximality of $(x_\epsilon,y_\epsilon)$, we have
\begin{equation}
    \begin{cases}
        &-J_p({x_\varepsilon},\dot \gamma_{\varepsilon}(0))\in d_{x_\varepsilon}^{+}u_0\\
        &-J_p({y_{\varepsilon}},\dot\gamma_{\varepsilon}(\varepsilon)) \in d^{-}_{y_\varepsilon}u_1,
    \end{cases}
\end{equation}
where $\gamma_{\varepsilon}:[0,\varepsilon]\to M$ is a minimizing geodesic connecting $x_{\varepsilon}$ and $y_{\varepsilon}$ in time $\varepsilon$.
We note $\|\dot\gamma_{\varepsilon}(0)\|_{x_{\varepsilon}}=\|\dot\gamma_{\varepsilon}(\varepsilon)\|_{y_{\varepsilon}}=\frac{d(x_{\varepsilon},y_{\varepsilon})}{\varepsilon},\quad  \forall \varepsilon>0;$ 
The definition of viscosity sub and super solution gives
\begin{align*}
u_0(x_{\varepsilon}) - u_1(y_{\varepsilon}) 
&\leq -H\! \left(x_{\varepsilon},-J_p({x_\varepsilon},\gamma_{\varepsilon}(0))\right)
   + H\!\left(y_{\varepsilon},-J_p({y_{\varepsilon}},\dot\gamma_{\varepsilon}(\varepsilon)) \right)
   + F_{0}(x_{\varepsilon}) - F_{1}(y_{\varepsilon}) \\[6pt]
&\leq C\Big(1 + \|\dot\gamma_{\varepsilon}(0))\|_{x_{\varepsilon}}^{p-1} 
               + \|\dot\gamma_{\varepsilon}(\varepsilon))\|_{y_{\varepsilon}}^{p-1}\Big)D_S((x_{\varepsilon},\dot \gamma_{\varepsilon(0)}), (y_{\varepsilon},\dot \gamma_{\varepsilon(\varepsilon)})) + F_{0}(x_{\varepsilon}) - F_{1}(y_{\varepsilon})\\         
&\leq C\Big(1 + \|\dot\gamma_{\varepsilon}(0))\|_{x_{\varepsilon}}^{p-1} 
               + \|\dot \gamma_{\varepsilon}(\varepsilon))\|_{y_{\varepsilon}}^{p-1}\Big)\\
&\qquad\cdot\Big( d(x_{\varepsilon},y_{\varepsilon}) 
     + \|\dot\gamma_{\varepsilon}(0) - \mathcal{T}_{y_{\varepsilon}}^{x_{\varepsilon}}[\gamma_{\varepsilon}](\dot\gamma_{\varepsilon}(\varepsilon))\|_{x_{\varepsilon}}\Big) 
     + F_{0}(x_{\varepsilon}) - F_{1}(y_{\varepsilon}) \\[6pt]
&\underbrace{\leq}_{\eqref{PTgeodesic}, \eqref{e: isometrydual}} C\left(1 + \|\dot \gamma_{\varepsilon}(0)\|^{p-1}_{x_\varepsilon}+\|\dot \gamma_{\varepsilon}(\varepsilon)\|^{p-1}_{y_\varepsilon}\right) d(x_{\varepsilon},y_{\varepsilon})
     + F_0(x_{\varepsilon}) - F_{1}(y_{\varepsilon}) \\[6pt]
&\leq C\left(d(x_{\varepsilon},y_{\varepsilon}) 
     + \frac{d^p(x_{\varepsilon},y_{\varepsilon})}{\varepsilon^{p-1}}\right) 
     + F_0(x_{\varepsilon}) - F_{1}(y_{\varepsilon})\\
\end{align*}

We now use $\Phi^{\varepsilon}(x_{\varepsilon},y_{\varepsilon})\geq \Phi^{\varepsilon}(x,x)$ for every $x\in M$. In particular, for all $x \in M$ 
\begin{equation}
u_0(x)-u_1(x)\leq u_0(x_{\varepsilon})-u_1(y_{\varepsilon})-\frac{d^p(x_\varepsilon,y_\varepsilon)}{p\varepsilon^{p-1}}.
\end{equation}
Using the bound that we found for $u(x_\varepsilon)-u(y_{\varepsilon})$ we have 
\begin{align}\label{e:limitcomparison}
    u_0(x)-u_1(x)&\leq  C\left(d(x_{\varepsilon},y_{\varepsilon}) + \frac{d^p(x_{\varepsilon},y_{\varepsilon})}{\varepsilon^{p-1}}\right) 
      + \frac{d^p(x_\varepsilon,y_\varepsilon)}{p\varepsilon^{p-1}} + ( F_0(x_{\varepsilon})- F_{1}(y_{\varepsilon})).
\end{align}
Now, by regularity
$$\limsup_{\varepsilon \to 0} \big\{F_0(x_{\varepsilon})- F_{1}(y_{\varepsilon})\big\}\leq F_0(\bar x)- F_1(\bar x)\leq \max_{M} F_0-F_1, \quad \forall x \in M.$$

Therefore passing to the limsup in \eqref{e:limitcomparison}, and recalling the vanishing behavior of the penalizations,
$$
u_0(x)-u_1(x)\leq \max_{M} F_0-F_1, \quad \forall x \in M.
$$
The claim.\\

When the manifold $M$ is not compact, the same strategy can be carried out following the strategy of \cite{ishii1984uniqueness} by adding localizing perturbations, and we only sketch the proof here. A possible choice for the perturbation, for $\delta>0$, is 
\[
m_{\delta,x_0}(x) = \delta\, d^{p}(x,x_0),
\]
where  $x_0\in M$ is a given point. Note that this perturbation is super-differentiable, and not smooth in general; moreover, its super-differential can be explicitly characterized by Proposition~\ref{superdifferentiability}. Using Assumption~\ref{assumption: Hamiltonian}, one can then easily conclude the argument.
\end{proof}

\subsection{The case of a more geometric Hamiltonian}
We now employ the announced strategy of proof in the case of a dissipative and reversible geometric Hamiltonian.

\begin{proposition}\label{p: comparison} Let $L:TM \to \R$ be a weak, reversible, and dissipative Lagrangian. Let $F_{i}: M \to \R, i=0,1$ be two proper function: $F_0$ bounded from above and l.s.c.; $F_1$ bounded from below u.s.c.
Let $u_i: [0,T]\times M \to \R$, $i=0,1$ be bounded from above subsolution and bounded from below supersolution, respectively, of the following
\begin{equation*}
    w(x)+H(x,d_xw)=F_{i}(x), \quad  x \in M
\end{equation*}
with $H(x,p)=\sup_{v\in T_xM}p(v)-L(x,v)$.
In addition, assume $u_0$ and $u_1$ to have compact sub- and super-levels, respectively.
Then,
$$\sup_{ M}u_0-u_1\leq \sup_{ M} F_0-F_{1}.$$
\end{proposition}

\begin{proof}[Proof of Theorem \ref{p: comparison}]
Set
\begin{equation}
    \Phi^{\varepsilon} (x,y):=u_0(x)-u_1(y)-D(\varepsilon,x,y).
\end{equation}

Let $(x_\varepsilon,y_{\varepsilon})_{\varepsilon>0} \subset M\times M$ be a sequence of maximum points for $\Phi^{\varepsilon}$. Such a sequence exists thanks to the compactness of the sub/super level sets and the u.s.c. of the function $\Phi^{\varepsilon}$.
By Lemma \ref{Convergence Maxim}, we deduce the existence of a subsequence  of $((x_\varepsilon,y_\varepsilon))_{\varepsilon}$ (that we do not relabel), converging to a maximum point of $(\bar x,\bar x)\in M$ of the difference $u_0-u_1$ over $M$. 
Then, recalling the super-differentiability of $D(\eps,\cdot,\cdot)$, it follows that 
\begin{equation}
    \begin{cases}
        &-\frac{\partial}{\partial v}L(x_{\varepsilon},\dot\gamma_{\varepsilon}(0))\in \partial^{+}u_0(x_\epsilon)\\
        &-\frac{\partial}{\partial v}L(y_{\varepsilon},\dot\gamma_{\varepsilon}(\varepsilon)) \in \partial^{-}u_1(y_\varepsilon).
\end{cases}
\end{equation}

Moreover, we have
\begin{align*}
H(x_{\varepsilon}, -\frac{\partial}{\partial v} L(x_\varepsilon, \dot{\gamma}_\varepsilon(0))) - H(y_\varepsilon, -\frac{\partial}{\partial v} L(y_\varepsilon, \dot{\gamma}_\varepsilon(\varepsilon)))&\underbrace{=}_{\eqref{Symmetry Hamiltonian}}
H(x_{\varepsilon}, \frac{\partial}{\partial v} L(x_\varepsilon, \dot{\gamma}_\varepsilon(0))) - H(y_\varepsilon, \frac{\partial}{\partial v} L(y_\varepsilon, \dot{\gamma}_\varepsilon(\varepsilon)))\\
&= \hat{H}(x_{\varepsilon}, \dot{\gamma}_{\varepsilon}(0))) - \hat{H}(y_{\varepsilon}, \dot{\gamma}_{\varepsilon}(\varepsilon)))\underbrace{=}_{\text{Theorem} \, \ref{t:EL}}0
\end{align*}

Thus, by definition of viscosity sub and super solutions 
\begin{equation*}
    u_0(x_{\varepsilon})-u_1(y_\varepsilon)\leq -H(x_\varepsilon,\frac{\partial}{\partial v}L(x_{\varepsilon},\dot\gamma_{\varepsilon}(0)))+H(y_{\varepsilon},-\frac{\partial}{\partial v}L(y_\varepsilon,\dot\gamma_{\varepsilon}(\varepsilon)))+ F_{0}(x_\varepsilon)- F_{1}(y_\varepsilon)=F_{0}(x_\varepsilon)- F_{1}(y_\varepsilon).
\end{equation*}
Therefore, by $\Phi^{\varepsilon}(x_{\varepsilon},y_{\varepsilon})\geq \Phi(\bar x,\bar x)$, we obtain 
$$
u_0(\bar x)-u_1(\bar x)\leq F_{0}(x_\varepsilon)- F_{1}(y_\varepsilon)-D(\varepsilon,x_{\varepsilon},y_{\varepsilon})\leq F_{0}(x_\varepsilon)- F_{1}(y_\varepsilon) 
$$
Sending $\varepsilon \to 0$, and using the regularity properties of the functions, we get 
$$
\max_{M} u_0-u_1 \leq F_{0}(\bar x)-F_1(\bar x)\leq \max_M F_0 -F_1.$$

\end{proof}
\begin{remark}[Possible interpretation of the proof]\label{r: Lyapunov Property}The standard proof uses the quadratic penalization function 
$-\frac{d^2(x,y)}{2\varepsilon}$
instead of the function \( D(\varepsilon, x, y) \). With our choice of penalization, the doubling variables argument shows that \emph{the function $ \big(u_0-F_0\big) - \big(u_1-F_1 \big)$ serves as  Lyapunov function for the flow $ \varepsilon \mapsto (x_\varepsilon, y_\varepsilon )$}. This is in contrast with the quadratic penalization, where this Lyapunov property holds only asymptotically as \( \varepsilon \to 0 \).

This interpretation is not new in the literature. It closely resembles the Lyapunov property described in the monograph \cite[Thm ~4.3.9]{fathi2008weak}, where the minimizing curve \( \gamma_\varepsilon \) selected in the doubling variables argument plays a role analogous to that of a calibrated curve. 

This argument also echoes the framework of \emph{generalized gradient systems} discussed in \cite[Section~3.2]{mielke2023introductionanalysisgradientssystems}, where dissipative properties play a central role in the analysis of gradient flows.
\end{remark}

\section{Mean Field Setting}\label{s: PrelWass}
In this section, we apply the strategy used in the previous section to cases in which $M$ is replaced by a Wasserstein space, which we shall formally see as a Riemannian manifold itself. For more details on such a formal analogy, we refer the reader to \cite{Gigli}.
For further material and details on  Wasserstein spaces with underling space Riemannian manifold, the reader can refer to the monograph \cite{villani2008optimal}. 
Let $X$ be a Polish space, and denote by $\mathcal{P}(X)$ the space of Borel probability measures on $X$.  
For $Y$ another Polish space and $\mu_0 \in \mathcal{P}(X), \mu_1 \in \mathcal{P}(Y)$, we define the set of \emph{couplings} between them as
\[
\Gamma(\mu_0,\mu_1)
:= \big\{ \gamma \in \mathcal{P}(X\times Y) : (p_i)_{\#}\gamma = \mu_i,\; i=1,2 \big\},
\]
where $p_{1/2} : X\times Y \to X/Y$ denotes the canonical projection onto the $i$-th component.

\smallskip
Fix $p>1$ and let $d : X \times X \to [0,\infty)$ be a distance on $X$.  
For a reference point $x_0 \in X$, we define
\[
\mathcal{P}_p(X)
:= \Big\{ \mu \in \mathcal{P}(X) : 
\mathcal{M}_p(\mu) := \int_X d^p(x,x_0)\, \mathrm{d}\mu(x) < \infty \Big\}.
\]
The space $\mathcal{P}_p(X)$, endowed with the \emph{$p$-Wasserstein distance}
\[
W_p(\mu_0,\mu_1)
:= \min_{\gamma \in \Gamma(\mu_0,\mu_1)}
\Bigg( \int_{X\times X} d^p(x,y)\, \mathrm{d}\gamma(x,y) \Bigg)^{1/p},
\]
is called the \emph{$p$-Wasserstein space} over $X$.

\smallskip
By the triangle inequality, the definition of $\mathcal{P}_p(X)$ does not depend on the choice of the reference point $x_0 \in X$.  
For a detailed account of the main properties of the space $(\mathcal{P}_p(X),W_p)$, we refer to \cite[Ch.~6]{villani2008optimal}.

\smallskip
In what follows, we will consider the cases $X = M$, $TM$, or $T^*M$.\\

\subsection{Couplings and notation}
In order to use the analogy which consists in seeing $\mathcal P_p(M)$ as a manifold, we need to use several notations which may seem arbitrary complex at first sight, but which will be of great use in what follows.\\

We firstly lift the exponential function associated with the underlying Riemannian structure
\begin{align}
    \exp_{\mathcal P} &: \mathcal{P}(TM) \to \mathcal{P}(M)\\
    &\sigma \mapsto \exp_{\#}\sigma. \notag
\end{align}
We note that $\exp_{\mathcal{P}}((\delta_x,\delta_v))=\exp_{x}(v)$.
Analogously, we can define $\exp^L_{\mathcal P}(\varepsilon;\cdot)$ as the lift of $\exp^L(\varepsilon;\cdot)$, whenever the latter is well defined, see \eqref{eq:2.16}. 

The exponential map suggests the introduction of the following 
\begin{align*}
    \Gamma_{TM}(\mu_0,\mu_1)&:=\big\{\sigma \in \mathcal{P}(TM) : (\text{Id},\exp)_{\#}\sigma \in \Gamma(\mu_0,\mu_1) \big\}\\
    \Gamma^{TM,L}(\varepsilon;\mu_0,\mu_1)&:=\big\{\sigma \in \mathcal{P}(TM) : (\text{Id},\exp^L(\varepsilon;\cdot))_{\#}\sigma \in \Gamma(\mu_0,\mu_1) \big\},
\end{align*}

that we call sets of \emph{exponential couplings} associated to the Riemannian and Lagrangian structure, respectively.

As already remarked by Gigli in \cite[p. ~14]{Gigli} , in the case the exponentials fails to be a global diffeomorphism, these sets give more information than the usual coupling: indeed, an exponential coupling not only specifies the transport between the unit of masses, but also an initial direction, distinguishing between the various possible source-target path assignments, which is of course not needed when $M = \R^d$.

We introduce several spaces that we shall encounter
\begin{align*}
    \mathcal{P}_{p}(TM) 
    &:= \Big\{ \sigma \in \mathcal{P}(TM) : \pi_{\#}\sigma \in \mathcal{P}_p(M), \,
    \|\sigma\|_{\pi_{\#}\sigma,p} := \Big(\int_{TM} \|v\|^{p}_x \, \sigma(dx,dv)\Big)^{\frac{1}{p}} < \infty \Big\}, \\
    \mathcal{P}_{p,q}(T^{*}M) 
    &:= \Big\{ \gamma \in \mathcal{P}(T^{*}M) : \pi^*_{\#}\gamma \in \mathcal{P}_p(M), \,
    {\|\gamma\|_{*}}_{\pi^*_{\#}\gamma,q} := \Big(\int_{T^*M} {\|z\|_{*}}^{q}_x \, \gamma(dx,dz)\Big)^{\frac{1}{q}} < \infty \Big\}.
\end{align*}

We also note that $\exp_{\mathcal{P}}(\mathcal{P}_p(TM))=\mathcal{P}_p(M)$. Indeed, 

\begin{align*}
\int_{M}d^p(x_0,x)d(\exp)_{\#}\sigma(x,v)&=\int_{TM}d^p(x_0,\exp_x(v))d\sigma(x,v)\\
&\leq 2^{p-1}\Big(\int_{TM}d^p(x_0,x)d\sigma(x,v)+\int_{TM}d^p(x,\exp_x(v))d\sigma(x,v)\Big)\\
&\leq 2^{p-1}\Big(\Big(\int_{TM}d^p(x_0,x)d\pi_{\#}\sigma(x)+\int_{TM}\|v\|^p_xd\sigma(x,v)\Big)\Big).
\end{align*}

Given 
two measures $\gamma \in \mathcal{P}_{p,q}(T^*M), \sigma\in \mathcal{P}_p(TM)$ we set
\begin{align*}&\Gamma_{q,p}(\gamma,
\sigma)=\Big\{\bm \sigma \in \mathcal{P}(T^*M \times TM) \,\big| \,(\pi_{12})_{\#}\bm\sigma=\gamma, (\pi_{3,4})_{\#}\bm \sigma=\sigma \Big\},\\
& \Gamma_{\mu}(\gamma,
\sigma):=\Big\{\bm \sigma \in \Gamma_{q,p}(\gamma,\sigma) \,\vert\, (\pi^{*},\pi)_{\#}\bm \sigma=(Id,Id)_\#\mu \Big\},
\end{align*}
where $\pi_{12}: T^*M \times TM \to T^*M$, $\pi_{34}: T^*M \times TM \to TM$ denote the natural projections.
Sometimes we will use the notation $\gamma_{\mu}:= \gamma \in \mathcal{P}(T^{*}M)$, whenever $\pi_{\#}\gamma= \mu$ and we will note, in this case, $\gamma_{\mu}\in \mathcal{P}_{\mu}(T^{*}M)$. We define $\sigma_{\mu}$ analogously.

\begin{remark}Fix $\mu \in \mathcal{P}_{p}(M)$ and $\gamma \in \mathcal{P}_{p,q}(T^*M), \sigma\in \mathcal{P}_p(TM)$.
We observe that $\bm \sigma \in \Gamma_{\mu}(\gamma,\sigma)$ can be identified with its disintegration $ d\zeta_x(v,z)d\mu(x)$, where $\zeta_x \in \mathcal{P}(T_xM \times T^{*}_xM) $ is $\mu$ a.e. uniquely determined by 
$$\int_{TM\times T^{*}M}f(x,v,x,z)d\bm\sigma(x,v,x,z)=\int_{M}\Big(\int_{T_xM\times T^*_xM}f(x,v,x,z)d\zeta_x(v,z)\Big)d\mu(x), \quad \forall f\in C_c(TM\times T^*M).$$
\end{remark}

Let $
\bm \sigma \in \Gamma_{\mu}(\gamma_{\mu},\sigma_{\mu})$ and $\zeta_x\mu$ its disintegration. We set
\begin{equation}\label{Actiondual}    (\gamma_\mu,\sigma_\mu)_{\bm\sigma}:=\int_{TM} z(v)\bm\sigma(dx,dz,dx,dv)=\int_{M}\Big(\int_{T_xM \times T^*_xM}z(v)d\zeta_x(z,v)\Big)d\mu(x).
\end{equation}
The previous quantity is finite by the Young Inequality, provided that  $\gamma \in \mathcal{P}_{p,q}(T^{*}M)$ and $\sigma \in \mathcal{P}_p(TM)$: Indeed, 
$$
(\gamma_\mu,\sigma_\mu)_{\bm\sigma}\leq \int_{TM\times T^*M}\Big(\frac{{\|z\|^{*}}^q_x}{q}+ \frac{{\|v\|}^p_x}{p}\Big)d \bm \sigma(x,z,x,v)\leq \frac{{\|\gamma\|_{*}}^q_{\mu,q}}{q}+ \frac{\|\sigma\|^p_{\mu,p}}{p}.
$$
This extends the duality of the underlying spaces as the case $\gamma_{\delta_{x}}=\delta_x \otimes  \delta_{v}$ and $\sigma_{\delta_{x}}=\delta_x \otimes  \delta_{z}$
shows 
\begin{equation*}
    (\gamma_{\delta_{x}},\sigma_{\delta_{x}})_{\bm \sigma}=z(v),
\end{equation*}
being $\bm \sigma \in  \Gamma_{\delta_x}(\delta_v, \delta_{z})=\Big\{(\delta_x \otimes\delta_v) \otimes (\delta_x \otimes \delta_{p})\Big\}$.
\\

\subsection{Growth conditions and Relaxed Hamiltonian}
In view of the structure of \eqref{eq: WHJ}, it is natural to assume growth bounds on the Hamiltonian $H$ in which we shall be interested, namely to ensure that the integral of the Hamiltonian is well defined. We here assume that there exists $p>1$, $c, C>0$, $x_0 \in M$ s.t. for al $(x,z) \in T^*M$,
\begin{align}\label{HamiltonianGrowth}
| H(x,z)|\leq C(1+ d^{p}(x,x_0))+c\frac{{\|z\|_{*}}^{q}_x}{q},
\end{align}
where $q$ is the conjugate exponent to $p$. Up to change of the constant $C>0$ the bound \eqref{Lagrangian growth} does not depend on the choice of $x_0\in M$.
When interested in more geometric problems, we shall make the following growth assumption on the associated Lagrangian. There exists $p>1$, $c, C>0$, $x_0 \in M$ s.t. for all $(x,v) \in TM$.
\begin{align}\label{Lagrangian growth}
\frac{1}{c}\frac{\|v\|^{p}_x}{p}-C(1+ d^{p}(x,x_0))\leq L(x,v)&\leq C(1+ d^{p}(x,x_0))+c\frac{\|v\|^{p}_x}{p}\\
 \label{growthL}   \left\|\frac{\partial L}{\partial v}(x,v)\right\|^{*}_x&\leq C(1+ \|v\|^{p-1}_x).
\end{align}

In the same spirit as Kantorovich's relaxation in optimal transport, the second author, in \cite[p. 17]{bertucci2024stochasticoptimaltransporthamiltonjacobibellman}, introduced a notion of relaxed Hamiltonian. In our setting,
the \emph{relaxed Hamiltonian} associated with $H$ is the function
$\textbf{H}:\mathcal{P}_{p,q}(T^{*}M) \to \R$ defined as
\begin{equation}\label{e: relaxedH}
    \textbf{H}(\gamma):=\int H(x,z)\gamma(dx,dz).
\end{equation}
Analogously, one can also define the notion of \emph{relaxed Lagrangian} as a function
$\textbf{L}:\mathcal{P}_{p}(TM) \to \R$ s.t.
\begin{equation}
    \textbf{L}(\sigma):=\int L(x,v)\sigma(dx,dv).
\end{equation}

We note that the growth conditions on the Lagrangian and the Hamiltonian ensure that $\textbf{H}$ and $\textbf{L}$ are well defined.
We suppose also the following \emph{$p$-superlinearity} on $L$: for every $K\geq 0$, there exists $C(K)\in \R$ s.t.
\begin{equation*}
    \forall (x,v)\in TM, \quad  L(x,v)\geq K \|v\|^{p}_x-C(K). 
\end{equation*}
Under this assumption, the relaxed Lagrangian satisfies:
for every $K\geq 0$, there exists $C(K)\in \R$ s.t.
\begin{equation*}
    \forall \sigma \in \mathcal{P}_p(TM), \quad  \text{L}(\sigma)\geq K \|\sigma\|^{p}_{\mu,p}-C(K). 
\end{equation*}

\subsection{ Fenchel Duality in  $\mathcal{P}(TM)$}
In this section we place ourselves in the framework of Section \ref{sec:hypgeo}. The duality extends also in this framework. Indeed, we have the lift of the Legendre Transform of \eqref{legendre} as
\begin{align}
    \mathcal{L_P}: \mathcal{P}(TM)& \to \mathcal{P}(T^*M)\\
    &\sigma \mapsto  \mathcal{L}_{\#}\sigma.
\end{align}
Note that it is invertible, and $\mathcal{L}_{\mathcal P}^{-1}=(\mathcal{L}^{-1})_{\#}$.

\begin{example}\label{ex:mathJ}
In the case of the Examples \ref{ex: p-norm} and \ref{Lagrangian&Geo} we adopted a different notation for the Legendre transform and we remain consistent with it in the MF setting. Here, $ \text{L}(\sigma)=\frac{\|\sigma_\mu\|^p_{\mu}}{p}.$ We will denote by $\mathcal{J}_p=\mathcal{L}_{\mathcal{P}}$ the lift of the Legendre Transform. We note that $\mathcal{J}_p=(J_p)_{\#}.$ 
    
\end{example}
With this choice, we have the following 
\begin{proposition}\label{p: WFenchel}
Let $L: TM \to \R$ be a weak Lagrangian satisfying the p-growth assumption. Denote by $q>1$ the conjugate exponent to $p$. Then,
\begin{enumerate}
\item \label{p:Duality 1}$\mathcal{L}_{\mathcal{P}}(\mathcal{P}_p(TM))=\mathcal{P}_{p,q}(T^{*}M).$
    \item \label{p:Duality 2}The relaxed Hamiltonian is the Fenchel transform of the relaxed Lagrangian, i.e. $\forall \mu \in \mathcal{P}_{p}(M)$ and $\gamma_{\mu}\in \mathcal{P}_{p,q}(T^*M)$
\begin{align*}
   \textbf{\emph{H}}(\gamma_{\mu})
   &=\sup_{\left\{ \bm \sigma \in \Gamma_{\mu}(\gamma_{\mu}, \sigma_{\mu})\, \vert \,\sigma_{\mu} \in \mathcal{P}_{p}(TM) \right\}} \left[ \int_{TM \times T^{*}M} z(v) \, d\bm{\sigma}(x, z, x, v) - \int_{TM}L(x,v)d\sigma_{\mu}(x,v) \right]\\
   &=\sup_{\left\{ \bm \sigma \in \bm \Gamma_{\mu}(\gamma_{\mu}, \sigma_{\mu})\, \vert \,\sigma_{\mu} \in \mathcal{P}_{p}(TM), \, \bm \sigma= \zeta_x \mu \right\}} \left[ \int_M \Big(\int_{T_xM \times T_x^{*}M}\big( z(v) \,  - L(x,v)\big)d\zeta_{x}(z,v)\Big)d\mu(x)\right]
\end{align*}

\end{enumerate}
\end{proposition}
\begin{proof}
We proceed in order.

\eqref{p:Duality 1}: By \eqref{growthL}, 
the following 
$$\int_{T^*M}{\|z\|_x}_{*}^{q}d\mathcal{L}(\sigma)(x,z)=\int_{TM}{\left\|\frac{\partial L (x,v)}{\partial v}\right\|_x}_{*}^{q}d\sigma(x,v)\leq \int_{TM} C(1+ \|v\|^{(p-1)q}_x)d\sigma(x,v)=C(1+ \|\sigma\|^p_{\mu,p})$$
holds for every choice of $\sigma \in \mathcal{P}_{p}(TM)$. This proves the assertion, since the Lagrange transform leaves the first marginal invariant and its invertibility implies the surjectivity.

\eqref{p:Duality 2}:
Fix $\mu \in \mathcal{P}_p(M)$, $\gamma_{\mu} \in \mathcal{P}_{p,q}(T^*M), \sigma_{\mu} \in \mathcal{P}_{p}(TM)$ and $\bm \sigma \in \bm \Gamma_{\mu}(\gamma_{\mu},\sigma_{\mu})$. 
Then
\begin{align*}
    \textbf{H}(\gamma_{\mu})&=\int H(x,p)d\gamma=\int_{TM\times T^*M} H(x,p)\bm \sigma(dx,dp,dx,dv) \\
    &\geq \int_{TM\times T^*M}  (p(v) - L(x,v))\bm \sigma(dx,dp,dx,dv),
\end{align*}
and the equality holds iff $\bm\sigma$ is concentrated on the graph $S=\big \{(x,z,x,v) \in T^{*}M \times TM \, \big |  \, \,p(v)=\mathcal{L}(x,v) \big\}$. Such a $\bm \sigma$ is necessarily of the form $\bm\sigma=d\delta_{(\frac{\partial L(x,v)}{\partial v})^{-1}}(v)d\gamma_x(z)d\mu(x).$
In particular, $\bm \sigma$ is an optimizer iff
$$\text{H}(\gamma_{\mu})+ \text{L}(\sigma_\mu)= \int_{TM}\frac{\partial L(x,v)}{\partial v}(v)d\sigma(x,v) \quad \text{and } \quad \mathcal{L}_P(\sigma_\mu)=\gamma_\mu.$$

\end{proof}

\subsection{A geometric distance for the doubling of variables}
Let $L$ be a Tonelli Lagrangian satisfying the growth conditions \eqref{Lagrangian growth} for some $p>1$. In order to introduce the appropriate penalization for the doubling of variables as well as several of its properties, we need to consider three transport-type problems. Let $\mu_0,\mu_1 \in \mathcal{P}_p(M)$, we consider

$$
    AC^p([0,\varepsilon]; M)=\big\{\gamma\in C([0,\varepsilon];M):  \int_0^\varepsilon \|\dot \gamma(s)\|_{\gamma(s)}^p ds < \infty \big\}.
$$

We then define
$$
    \text{Geo}_{L}(\varepsilon;M):=\Bigl\{ \gamma \in AC^p([0,\varepsilon];M): \gamma \text{ minimizer of } \mathcal{A}_{[0,\varepsilon]}  \Bigr\} \underbrace{\subseteq}_{\text{Theorem } \ref{t:EL}} C^1([0,\varepsilon]; M),
$$
the set of minimizing path of the action associated to $L$, namely $\mathcal{A}_{[0,\varepsilon]}(\gamma)=\int_0^\varepsilon L(\gamma(s),\dot\gamma(s))ds$. 
We consider the following variational problem
\begin{equation}  \label{dynamic}
    D^{\text{Dyn}}(\varepsilon,\mu_{0},\mu_1):=\inf_{\eta  \in \Gamma^{\text{Dyn}}(\mu_0,\mu_1)} \Bigl\{\int_{AC^p([0,\varepsilon]; M)}\int_{0}^{\varepsilon}L(\gamma(s),\dot\gamma(s))dsd\eta(\gamma)\Bigr\},
\end{equation}
where $\Gamma^{\text{Dyn}}(\varepsilon;\mu_0,\mu_1)=\Big\{  \eta\in \mathcal{P}(AC^p([0,\varepsilon];M) \, | \, (e_{_i})_{\#}\eta=\mu_{i/\varepsilon}, \, i=0,\varepsilon\Big\}.$ Here, $e_t : C([0,\varepsilon];M)\to M, \,  t \in [0,\varepsilon]$ denotes the evaluation map, i.e. $e_t(\gamma)=\gamma(t)$. This is the dynamic Lagrangian formulation of an optimal transport problem.

We also introduce
\begin{align}\label{def: D}
    D_M(\varepsilon,\mu_0,\mu_1):&=\min_{\pi \in \Gamma(\mu_0,\mu_1)}\Bigl\{\int D(\varepsilon,x,y)\pi(dx,dy)\Bigr\}.\\
    \label{D: TML} D_{TM,L}(\varepsilon,\mu_0,\mu_1):&=\min_{\sigma \in \Gamma_{TM,L}(\varepsilon;\mu_0,\mu_1)}\Bigl\{\int D(\varepsilon,x,\exp^{L}_x(\varepsilon;v))\sigma(dx,dv)\Bigr\}\\
     D_{TM}(\varepsilon,\mu_0,\mu_1):&=\min_{\sigma \in \Gamma_{TM}(\mu_0,\mu_1)}\Bigl\{\int D(\varepsilon,x,\exp_x(v))\sigma(dx,dv)\Bigr\}
\end{align}

The equivalence of the three previous problems is somehow standard in the literature and we state and prove it mainly for the sake of completeness.

\begin{theorem}\label{t: equivalent}
Fix $\mu_0,\mu_1 \in \mathcal{P}_{p}(M)$. 
Then 
\begin{enumerate}
\item\label{i: 0} Given $\pi \in \Gamma(\mu_0,\mu_1)$, there exist $\sigma\in \Gamma_{TM}(\mu_0,\mu_1)$ and $\sigma_{L} \in \Gamma_{TM,L}(\mu_0,\mu_1)$ s.t. 
\begin{align*}
    \int_{TM}D(\varepsilon,x,\exp_{x}(v))d\sigma(x,v)=\int_{TM}D(\varepsilon,x,\exp^L_{x}(v))d\sigma_L(x,v)=\int_{M \times M}D(\varepsilon,x,y)d\pi(x,y).
\end{align*}
In particular, 
\begin{equation}
    D_{TM}(\varepsilon,\mu_0,\mu_1)=D_{TM,L}(\varepsilon,\mu_0,\mu_1)=D_M(\varepsilon,\mu_0,\mu_1).
\end{equation}
    \item \label{i:1} 
    If $\pi $ is an optimal coupling for $D_M(\varepsilon,\mu_0,\mu_1)$, then there exists a measurable map $G : M \times M \to \mathrm{Geo}_{L}(\varepsilon;M)$
    assigning to each pair $(x,y) \in \mathrm{supp}(\pi)$ a minimizing Lagrangian geodesic joining $x$ and $y$.  
    Moreover, the push-forward measure $\eta := G_{\#}\pi \in \Gamma^{\mathrm{Dyn}}(\varepsilon;\mu_0,\mu_1)$
    is a minimizer of the dynamic formulation.  
    Conversely, if $\eta \in \Gamma^{\mathrm{Dyn}}(\varepsilon;\mu_0,\mu_1)$ is a minimizer of \eqref{dynamic}, then $\eta$ is concentrated on $\mathrm{Geo}_{L}(\varepsilon;M)$, and the induced transport plan $(e_0,e_{\varepsilon})_{\#}\eta \in \Gamma(\mu_0,\mu_1)$
    is optimal for the  cost $D(\varepsilon,x,y)$ whenever $\eta$ is optimal.

    \item \label{i:2} 
    If $\sigma \in \Gamma_{TM}(\mu_0,\mu_1)$, then the push forward measure 
    \[
        \eta_{\varepsilon} := (\exp^L)_{\#}\,(\mathrm{d}t \otimes \sigma)
    \]
    defines a dynamic plan in $\Gamma^{\mathrm{Dyn}}(\varepsilon;\mu_0,\mu_1)$.Moreover, $\eta$ is optimal whenever $\sigma$ is.
\end{enumerate}
\end{theorem}

\begin{proof}
We first note that $D(\varepsilon,\cdot,\cdot):M \times M \to \R$ satisfies the usual condition to get existence of an optimal coupling $\pi\in \Gamma(\mu_0,\mu_1)$, see \cite[Ch. ~4]{villani2008optimal}.

\eqref{i: 0}: The construction of both $\sigma$ and $\sigma_L$ is based on a measurable selection argument. We proceed with the existence of $\sigma_L$, the one of $\sigma$ being completely analogous.  

Let 
\begin{align*}
    F_L: &M \times M \rightrightarrows TM\\
    &(x,y) \longmapsto (x,v(x,y)), 
\end{align*}
where $\exp_{x}(v(x,y))=y$, be the multivalued map induced by the preimage of the  Lagrangian exponential map (that is surjective by completeness of the E-L Lagrangian Flow). Therefore, $F_{L}((x,y))\not= \emptyset$. Since $L\in C^2$, $\exp^L$ is continuous. In particular, 
$$\text{Graph}_{F_L}=\Big\{(x,y,x,v)\in M^2 \times TM: y=\exp_x^L(\varepsilon;v)\Big\}$$
is closed, and $F_L$ is Borel measurable multifunction. Therefore, by the Kuratowski-Ryll-Nardzewski Theorem \cite[Thm~6.9.3]{Bogachev2007}, for fixed $\pi \in \mathcal{P}(M \times M)$ there exists a Borel selection $(\text{Id}_M,v^L): \text{supp}(\pi) \to TM$ of $F_L$, s.t.  
$$
y=\exp_x(v^L(x,y)) \quad  (x,y)-\pi \,a.e$$
We conclude that 
$$\sigma^L:=(\text{Id}_M,v^L)_{\#}\pi \in \Gamma_{TM,L}(\mu_0,\mu_1) \quad  \text{ and }\quad  \int_{M\times M}D(\varepsilon,x,y)d\pi(x,y)=\int_{TM}D(\varepsilon,x,\exp^L_{x}v)d\sigma(x,v).$$

\eqref{i:1}: The equality $D(\varepsilon,\mu_0,\mu_1)=D^{\text{Dyn}}(\varepsilon,\mu_0,\mu_1)$ is the content of  \cite[Thm ~ 7.21 \& Rmk ~ 7.25]{villani2008optimal}. 
The same reference gives that whenever $\eta\in \Gamma^{\text{Dyn}}(\varepsilon;\mu_0,\mu_1)$ is a  minimizer, $\eta$ is concentrated on $\text{Geo}_L(\varepsilon;M)$. In particular, it is concentrated on solutions of the E-L equation that are of class $C^1$. We note that $\text{Geo}_{L}(\varepsilon;M)$ is a closed subset of $C^1([0,\varepsilon];M)$.

\eqref{i:2}:
We consider the flow map
\begin{align*}
\exp^{L}:\,&[0,\varepsilon]\times TM  \to \text{Geo}_L(\varepsilon;M)\\
&(t,x,v) \longmapsto \gamma,
\end{align*}
where $\gamma: [0,\varepsilon] \to M$ s.t. $\gamma(0)=x, \dot \gamma(0)=v$, and $\gamma$ solution of E-L equation. This map is well defined since the Lagrangian is assumed to be $C^2$. In particular, $\exp_{\gamma(0)}(t;\dot \gamma(0))=\gamma(t).$ This is a continuous map as consequence of the stability w.r.t. initial condition of the E-L flow, ensured by $L\in C^2$.
Define the plan $\eta=(\exp^L)_{\#}\text{d}t\otimes\sigma$, where $\sigma\in \Gamma_{TM}(\mu_0,\mu_1)$. Then $\eta \in \Gamma_{\text{Dyn}}(\mu_0,\mu_1)$, and 
$\eta$ is optimal whenever it is $\sigma$, by combining points \eqref{i:1} and \eqref{i:2}

\end{proof}

\subsection{Superdifferential calculus on $\mathcal{P}_p(M)$.}
We adopt the following notion of superdifferentability for functions on $\mathcal{P}_{p}(M)$.
\begin{definition}[Superdifferential]
Given an upper semi continuous function $U \colon \mathcal{P}_p(M)\to \R$, we say that $ \gamma_{\mu} \in \mathcal{P}_{p,q}(T^*M)$ is a \emph{superdifferential} of $U$ at the point $\mu$ if 
for any $\sigma \in \mathcal{P}_{p}(TM)$, for any $\bm \sigma \in \Gamma_{\mu}(\gamma_\mu,\sigma_\mu)$, 
\begin{equation}\label{def:superdifferentiability}
    U(\exp_{\mathcal{P}}(\sigma_{\mu}))-U(\mu)\leq (\gamma_{\mu},\sigma_{\mu})_{\bm \sigma}+ o(\Big({\int_{TM} d^p(x,\exp_x(v)) \sigma_{\mu}(dx,dv) }\Big)^{\frac{1}{p}}),
\end{equation}

In this case, we note $\gamma_{\mu} \in \partial^{+}U(\mu)$. 
Given a lower semi-continuous function $U \colon \mathcal{P}_p(M)\to \R$, we say that $\gamma_\mu \in \mathcal{P}_{p,q}(T^{*}M)$ is a \emph{subdifferential} of $U$ at the point $\mu$ if $(\pi,-\text{Id})_{\#}\gamma_{\mu} \in \partial^{+}(-U)(\mu)$. We note $\partial^-U(\mu)$ the set of such elements.
\end{definition}

\begin{remark}
This notion of super-differential is analogous to the one called extended super-differential in \cite{AGS} adapted in our setting.    
\end{remark}

Clearly, this definition coincides with the usual finite-dimensional one when $\gamma_{\delta_{x}}=\delta_x \otimes  \delta_{v}$, and \eqref{def:superdifferentiability} is tested along $\sigma_{\delta_{x}}=\delta_x \otimes  \delta_{p}$.

\begin{lemma}(Differentiability)
 Let $U: \mathcal{P}_p(M) \to \R$ be a continuous function. Let $\gamma^{+} \in \partial^{+} U(\mu)$ and $\gamma^{-} \in \partial^{-} U(\mu)$ for some $\mu \in \mathcal P_p(M)$, then $\gamma^{+}=\gamma^{-}$. In particular,  $\partial^{+}U(\mu) \cap \partial^{-}U(\mu) $ contains at most one element. 
\end{lemma}

\begin{proof}
Since $\gamma^+_\mu \in \partial^{+}U(\mu)$ and $\gamma^{-}_{\mu} \in  \partial^{-}U(\mu)$, for every $\sigma_{\mu}\in \mathcal{P}_{p}(TM)$ we have
\begin{equation*}
    (\gamma^{+}_{\mu},\sigma_{\mu})_{\bm \sigma^{+}}- (\gamma^{-}_{\mu},\sigma_{\mu})_{\bm \sigma^{-}}\leq o\left(\Big({\int_{TM} d^p(x,\exp_x(v)) \sigma_{\mu}(dx,dv) }\Big)^{\frac{1}{p}}\right))= o\left(\Big({\int \|v\|^{p}_x {}\sigma(dx,dv) }\Big)^{\frac{1}{p}}\right),
\end{equation*}
where $\bm \sigma^{\pm}=d\zeta^{\pm}_{\cdot}\mu\in \Gamma_{\mu}(\gamma^{\pm}_{\mu},\sigma).$
Fix $\zeta_x(z^+,z^-,v)=\zeta_{x}(z^{+},z^{-})\sigma_{x,z^{+},z^{-}}(v)$, where $\zeta_{x} \in \Gamma(\gamma^{+}_x,\gamma_x^{-})$. This procedure is well defined $\mu$ almost everywhere and we do not detail the classical fact that we can do it in a measurable way. It then follows 
\begin{align*}(\gamma^{+}_{\mu},\sigma_{\mu})_{\bm \sigma^{+}}- (\gamma^{-}_{\mu},\sigma_{\mu})_{\bm \sigma^{-}}
&=\int_{M}\Big(\int_{T_xM\times T_x^*M}z^+(v)d\zeta^+_x(z^{+},v)-z^{-}(v)d\zeta^{-}_x(z^{-},v)\Big)d\mu(x)\\
&=\int_{M}\Big(\int_{T_xM\times (T_x^*M)^2}(z^+-z^{-})(v)d\zeta_x(z^{+},z^{-},v)\Big)d\mu(x).
\end{align*}
Set $\tilde\gamma_x:= (\pi_x^{+}-\pi_x^-)_{\#}(\zeta_x)$, $\pi^{\pm}: T_xM\times (T_x^*M)^2 \to  T_x^*M$, are the projection on the second and the third factor. Consider $\sigma_{\mu}=\mathcal{J}^{-1}_{p}(\tilde \gamma_\mu)$, where $\tilde \gamma_\mu=\tilde \gamma_x \mu(dx)$ and $\mathcal J$ is defined in Example \ref{ex:mathJ}. It is then enough to prove $\tilde \gamma_\mu=\delta_0(dz)\mu$. Thanks to this choice of $\sigma_\mu$, using \eqref{dualitynorms}
$$\int_{M}\Big(\int_{T_xM\times (T_x^*M)^2}(z^+-z^{-})(v)d\zeta_x(z^{+},z^{-},v)\Big)d\mu(x)=\int_{T^*M}{\|z\|^{q}_{*}}_xd\tilde {\gamma}_{\mu}\leq o\Big((\int_{T^{*}M}{\|z\|^{q}_{*}}_xd\tilde \gamma_{\mu})^{\frac{1}{p}}\Big).$$

In particular,
$$0\leq {\|\tilde{\gamma}_{\mu}\|_*}_{\mu,q}\leq o(1),$$
which implies $\tilde \gamma_\mu=\delta_0(dz)\mu$, thus the result.

\end{proof}

\begin{definition}[Wasserstein Differential]
Fix $\mu \in \mathcal{P}_p(M)$. We denote by $d_\mu U \in \mathcal{P}_{p,q}(T^*M)$ the unique element of the intersection $\partial^{+} U(\mu)\cap \partial^{-} U(\mu)$, whenever the intersection is not empty. We refer to it as the \emph{Wasserstein differential of $U$ at $\mu$}.   
\end{definition}

The next result is essential in what follows. It states the super-differentiability of the penalization function $D_M$ that we introduced above and shall use in the geometric setting below. Unfortunately, it seems to require some geometric assumption, which is here stated as a form of integrability condition on the concavity bound on the distance $D$ (at the level of the manifold). We give examples where such a situation is satisfied after the proof.
\begin{proposition}\label{wsuperdifferentiability} 
Fix $\varepsilon>0$.
Let $L$ be a Tonelli Lagrangian satisfying the growth condition \eqref{Lagrangian growth}. 
Suppose in addition that the cost $D(\varepsilon,\cdot, \cdot): M \times M \to \R$ is $(\lambda,\omega)$-semiconcave  for $(\lambda,\omega)$ such that there exists $C(\lambda, \sigma)$ for $\|\sigma\|_{\pi_\#\sigma}$ such that
\begin{equation}\label{compatibilitysemiconcavity}
\int_{M\times M}\lambda(x,\exp_x(v))\omega(D(\varepsilon,x,\exp_x(v)))d\sigma(x,v)\leq C(\lambda, \sigma)\omega\left(\left(\int d^{p}(x,y)d\pi(x,y)\right)^{\frac{1}{p}}\right)
\end{equation}
holds for any $\pi \in \mathcal{P}_{p}(M\times M)$.

Let $\sigma_0 \in \Gamma_{TM,L}(\varepsilon;\mu_0,
\nu)$ be an optimal plan for $D_{TM}(\varepsilon,\mu_0,\nu)$, then
\begin{equation}
 \gamma(dx,dp):= (\mathrm{Id}\times (-\mathrm{Id}))_{\#}\mathcal{L}_{\mathcal{P}}(\sigma_0) \in \partial^{+}D_M(\varepsilon,\cdot,\nu)(\mu_0),
 \end{equation}
where we use the slight abuse of notation $(\mathrm{Id}\times (-\mathrm{Id}))(x,v) = (x,-v)$, for $(x,v) \in TM$.
\end{proposition}
 In particular, with the notation and assumptions of the Proposition, for all Borel maps $f: T^*M \to \R$, we have
 \begin{equation*}
     \int_{T^*M} f(x,p)\gamma(dx,dp)=\int_{TM} f\left(x,-\frac{\partial}{\partial v}L(x,v)\right)\sigma_0(dx,dv).
 \end{equation*}

\begin{proof}
Fix $\sigma_1 \in \mathcal{P}_p(TM)$ with $\pi_{\#}\sigma_1=\mu_0$, $\mu_1=\exp_{\mu_0}(\sigma_1) \in \mathcal{P}_p(M)$, and $\nu \in \mathcal{P}_p(M)$. Consider $\bm \sigma_{0 \to1}$ a three plan s.t. $\pi_{1,2}{\#}\bm \sigma_{0,1}=\sigma_0$, and $\pi_{1,3}{\#}\bm\sigma_{0,1}=\sigma_1\in \Gamma_{TM}(\mu_0,\mu_1)$. We note that, by definition, $\zeta:=(\exp_{\pi_1}(\pi_3), \exp_{\pi_1}(\pi_2))_{\#}{\bm \sigma_{0 \to 1}} \in \Gamma(\exp_{\mu_0}(\sigma_1),\nu)$.
Then
\begin{align*}
D_M(\varepsilon, \exp_{\mu_{0}}(\sigma_1),\nu)&-D_M(\varepsilon,\mu_0,\nu)\leq \int_{M\times M} D(\varepsilon,
x_1,y)\zeta(dx_1,dy)-\int D(\varepsilon,x_0,\exp^{L}_{x_0}(\varepsilon;v_0))\sigma_0(dx_0,dv_0)\\
&= \int \Bigl(D(\varepsilon,\exp_{x_0}(v_1),\exp^{L}_{x_0}(\varepsilon;v_0))- D(\varepsilon,x_0,\exp^{L}_{x_0}(\varepsilon;v_0))\Bigr)\bm\sigma_{0,1}(dx_0,dv_0,dv_1)\\
\underbrace{\leq}_{\text{Proposition } \ref{Differentiability Penalization}} & \int \Big(-\frac{\partial}{\partial v}L(x_0,v_0)  (v_1) + \lambda(x_0,\exp_{x_0}(v_1))\omega(d(x_0,\exp_{x_0}(v_1))) \Bigr)\bm\sigma_{0,1}(dx_0,dv_0,dv_1)\\
&= \int p (v_1)\bm \sigma(dx_0,dp,dv_1)+ C(\lambda,\sigma)\omega\Big((\int d^p(x,\exp_x(v))d\sigma_1(x,v)\big)^{\frac{1}{p}}\Big)
\end{align*}
where $\bm\sigma=(\mathcal{L}\times \text{Id})_{\#}\bm\sigma_{0 \to 1}\in \Gamma_{\mu_0}(\gamma, \sigma_1)$. 
In the second inequality, we used the fact that $\sigma_0$, being optimal, is concentrated on the set of point $(x_0,v_0)\in TM$ s.t. $v_0$ is the initial direction of an optimal curve connecting $x_0$ and $\exp_{x_0}^L(v_0)$, see Theorem \ref{t: equivalent}. 
\end{proof}

\begin{remark}\label{rem:superwp}
We observe that condition \eqref{compatibilitysemiconcavity} is fulfilled in all the situations described in Example \ref{ex: semiconcave}.
\end{remark}

We then state two lemmas that are the infinite dimensional analogous of Lemmas \ref{Decreasing} and \ref{Convergence Maxim}.

\begin{lemma}\label{WDecreasing}
Let $L$ be a dissipative Lagrangian, satisfying the growth assumptions \eqref{Lagrangian growth}. 
The function $D_M(\varepsilon, \cdot,\cdot): \mathcal P_p(M) \times \mathcal P_p(M) \to [0,\infty)$ defined in \eqref{def: D}
is non-negative and $\forall \mu,\nu \in \mathcal{P}_p(M)$
\begin{equation}
    D_M(\varepsilon,\mu,\nu)> D_M(\tau\varepsilon,\mu,\nu) \quad \forall \tau> 1.
\end{equation}
\end{lemma}
\begin{proof}
Fix $\varepsilon, \tau>0$, we denote by 
\begin{align*}
    h_\tau: &AC^p({[0,\varepsilon]; M)}\to AC^p({[0,\tau \varepsilon]; M)}\\
    & \gamma(t) \mapsto \gamma(\frac{t}{\tau}).
\end{align*}
Fix $\gamma \in AC^p({[0,\varepsilon]; M)}$, then  $\frac{d}{dt}h_\frac{1}{\tau}(\gamma)(t)=\frac{1}{\tau}\dot\gamma(\frac{t}{\tau})$ for a.e $t \in [0,\varepsilon]$. 
Let $\eta \in \Gamma^{\text{Dyn}}(\varepsilon; \mu_0,\mu_1)$ be an optimizer for the dynamic formulation, and fix $\tau>1$. Then, we have that $\eta$ is concentrated on $AC^p([0,\varepsilon];M)$ and
\begin{align*}
D_M(\varepsilon,\mu,\nu) &\geq \int_{AC^p([0,\varepsilon];M)}\int_0^{\varepsilon} L(\gamma(t), \dot{\gamma}(t)) \, dtd\eta(\gamma)\\
&= \int_{AC^p([0,\varepsilon];M)}\frac{1}{\tau} \int_0^{\tau \varepsilon} L(h_{\frac{1}{\tau}}(\gamma), \dot\gamma(\frac{t}{\tau})) \, dt  d\eta(\gamma)\\
&\underbrace{>}_{\text{strict convexity} \& L(\cdot,0)=0}  \int_{AC^p([0,\varepsilon];M)}\int_0^{\tau \varepsilon} L(h_{\frac{1}{\gamma}}(\gamma), \frac{1}{\tau}\dot \gamma(\frac{t}{\tau}))dt d\eta(\gamma)\\
&=\int_{AC^p([0,\varepsilon];M)}\int_0^{\tau \varepsilon} L(h_\frac{1}{\tau}(\gamma)(t), \frac{d}{dt}h_{\frac{1}{\tau}}(\gamma)(t)) \, dt d\eta(\gamma)\\
&= \int_{AC^p([0,\tau \varepsilon];M)}\int_0^{\tau \varepsilon} L(\gamma(t), \dot \gamma(t)) \, dt d{(h_{\frac{1}{\tau}}}_{\#}\eta)(\gamma)\\
& \geq D_M(\tau \varepsilon, \mu,\nu).
\end{align*}
\end{proof}

\begin{lemma}\label{WConvergence Maxim}

\begin{enumerate}
\item The (decreasing) sequence of continuous functions $D_M(t, \cdot,\cdot): \mathcal{P}_p(M)\times \mathcal{P}_p(M) \to \R$ is Gamma converging w.r.t to $W_p$ topology to the convex indicator function over the diagonal $\Delta \subset \mathcal{P}_p(M) \times \mathcal{P}_p(M)$ as $t \downarrow 0$.

\item Given $F: \mathcal{P}_p(M) \times \mathcal{P}_p(M) \to \R$  u.s.c. with compact superlevels

$$M_t:=\max_{\mu,\nu \in \mathcal{P}_p(M)\times \mathcal{P}_p(M)} \Big\{F(\mu,\nu)-D_M(t,\mu,\nu)\Big\} \  \quad\uparrow_{ t \downarrow 0} \quad M_0=\max_{\mu\in \mathcal{P}_p(M)} F(\mu,\mu),$$

and there exists a subsequence of $(\mu_\varepsilon,\nu_\varepsilon)\in \mathrm{argmax}_{(\mu,\nu)\in M}\big\{F(\mu,\nu)-D_M(\varepsilon,\mu,\nu)\big\}$ that converges in the $W_p$ topology to a point of maximum for $M_0$.

\end{enumerate}
\end{lemma}

\begin{proof}
Fix $\mu,\nu \in \mathcal{P}_p(M),$ and call $D_{t,M}:= D_M(t,\cdot,\cdot)$. Then
\begin{align*}
D_{t,M}(\mu,\nu)
&= \min_{\eta \in \Gamma^{\mathrm{Dyn}}(t;\mu,\nu)}
   \int_{AC^p([0,t];M)} \left( \int_0^{t} L(\gamma(s),\dot\gamma(s)) \, ds \right) \, d\eta(\gamma) \\
&\;\;\underbrace{\geq}_{\eqref{Lagrangian growth}}
   \;\frac{1}{cp} \inf_{\eta \in \mathcal{P}(C([0,t];M))}
   \int_{C([0,t];M)} \left( \int_0^t \|\dot\gamma(s)\|_{\gamma(s)}^p - Cd^p(\gamma(s),x_0)\, ds \right) d\eta(\gamma)
   \;-\; t\,C,
\end{align*}
where we used the notation of \eqref{Lagrangian growth}.

Now, choose $\eta\in \mathcal{P}(AC^p([0,t];M))$ concentrated on minimizing geodesics s.t. ${e_0}_{\#}\eta=\mu, {e_t}_{\#}\eta=\nu$. With this choice $\|\dot \gamma(s)\|_{\gamma(s)}=\frac{{d(\gamma(0),\gamma(t))}}{t}, \quad s \in [0,t]$ $\gamma-\eta \, a.e.$, and $\eta$ minimizes the RHS of the previous inequality of we omit the term in $C$.
We then infer
\begin{equation}
    D_{t,M}(\mu,\nu)\geq \frac{1}{cp} \frac{W^p_{p}(\mu,\nu)}{t^{p-1}}- tC(1 + \max (\|\mu\|_p^p , \|\nu\|_p^p)).
\end{equation}

Hence, the sequence $D_t$ is pointwise converging to $1_{\Delta}.$ 
Due to the monotonicity stated in Proposition \ref{Decreasing}, the family of continuous functions $D_{t,M}$
is increasing as $t \downarrow 0$. By \cite[Rmk~2.12]{bra06}, the $\text{Gamma}$-limit (w.r.t. to the $W_p$ topology) of $(D_{t,M})_{t>0}$ as $t \to 0$ coincides with the lower semicontinuous envelope of the pointwise limit, namely $1_\Delta$, which is already lower semicontinuous. This proves the first point. 

The rest of the proof is completely analogous to the one of the finite dimensional Lemma \ref{Convergence Maxim}.
\end{proof}

\subsection{Weak notions of convergence.}\label{weakconvergence}
The space $(\mathcal{P}_p(M), W_p)$ is not locally compact whenever the underlying space $M$ is non-compact (see \cite[Rmk.~7.1.9]{AGS}). This poses a difficulty in the proof of the comparison principle, where one typically needs to extract a \emph{maximizing} sequence. Thus, assuming compactness of level sets in this not locally compact case requires some explanation. A more systematic treatment of such questions is addressed in \cite{bertuccilions}, but we still give some details here for the sake of completeness.

A natural way to address this issue is to introduce a \emph{coercive penalization}, or to employ variants of the so-called \emph{Stegall Lemma} (see \cite{Stegall1978,Stegall1986}), which have already been used in the literature to handle the lack of compactness in optimization problems on Banach spaces (See for instance the proof of \cite[Thm.~1]{CrandallLions1984} for its use in the comparison principle). In the setting $(\mathcal{P}_2(\mathbb{R}^d), W_2)$, this strategy works, thanks to the so-called \emph{Hilbertian} (or \emph{Lagrangian}) lift --- as will be detailed in a forthcoming paper by the first author. 

However, extending this approach to $(\mathcal{P}_p(M), W_p)$ for a general Riemannian manifold $M$,even for $p=2$, is not straightforward, since the Lagrangian formulation no longer provides a linear structure. 

To overcome this difficulty, we shall make certain continuity assumptions on the functions we consider. We postpone to future works the detailed analysis of the interesting problem of dealing with the lack of local compactness in optimization over the Wasserstein space.

We say that a sequence $(\mu_n)_n \subset \mathcal{P}_p(M)$ \emph{weakly converges in $\mathcal P_p(M)$} to $\mu \in \mathcal{P}(M)$ if $\mu_n$ converges narrowly to $\mu$ and the $p$-th moments are uniformly bounded, i.e., $\sup_n \mathcal{M}_p(\mu_n) < \infty.$
By the lower semicontinuity of the moments with respect to narrow convergence, it then follows that
\[
\mathcal{M}_p(\mu) \le \liminf_{n\to\infty} \mathcal{M}_p(\mu_n) < \infty,
\]
so that $\mu \in \mathcal{P}_p(M)$. Of course, in general the first inequality is strict.

Moreover, if $\mu_n$ converges narrowly to $\mu$ and 
$\mathcal{M}_p(\mu_n) \to \mathcal{M}_p(\mu),$
then it follows that $W_p(\mu_n, \mu) \to 0,$ 
i.e., $\mu_n$ converges to $\mu$ in the Wasserstein topology. See \cite{villani2008optimal} for more details on such topologies.

Moreover, by the Banach-Alaoglu Theorem, the space $\mathcal P_{p}(M)$, endowed with this topology, is locally compact:
\begin{lemma}
Let $(\mu_n)\in \mathcal{P}_p(M)$ a sequence of measures s.t. $\sup_n \mathcal{M}_p(\mu_n)<\infty$. Then, there exists a $\mu \in \mathcal{P}_p(M)$
\begin{enumerate}
    \item $\mu_n$ weakly converges to $\mu$ in $\mathcal P_p(M)$.
    \item $\mu_n \to \mu$ in $\mathcal{P}_{p'}(M)$, for all $1\leq p'<p.$
\end{enumerate}
\end{lemma}
In particular, when we shall assume that a function is weakly lower semi-continuous in $\mathcal P_p(M)$ with bounded sub-level sets, then it reaches its minimum. Furthermore, any function which is lower semi continuous in $\mathcal P_{p'}(M)$ for $p' < p$ is  weakly lower semi-continuous in $\mathcal P_p(M)$. 

When the underlying manifold $M$ is compact, convergence in $W_p$ coincides with weak convergence in $\mathcal P_p(M)$. We also remark that the Wasserstein distance is lower semicontinuous with respect to this notion of convergence (see \cite[Prop.~7.1.3]{AGS}).

\section{First Order Hamilton-Jacobi equations in $\mathcal{P}_p(TM).$}\label{s: wComparison}

This section is devoted to the study of the following equation 

\begin{equation}\label{WHJB}
W(\mu)+\int_{M}H(\mu,x,d_\mu W)d\mu=\mathcal{F}(\mu), \quad  \mu \in  \mathcal{P}_p(M).
\end{equation}

Here, we assume the Hamiltonian $H$ to have the growth described in \eqref{HamiltonianGrowth} for some $q >1.$ 

A viscosity theory approach for the well-posedness for equation \eqref{WHJB} is natural. Again, we concentrate on the comparison principle.

The scope of this section is to introduce a notion of viscosity solutions in the Wasserstein framework. It involves the relaxed Hamiltonian \eqref{e: relaxedH} introduced at the beginning of the previous section, and was adopted in a flat compact case in \cite[Def.~2.8.]{bertucci2024stochasticoptimaltransporthamiltonjacobibellman}

\begin{definition}
  We say that an upper-semicontinuous function $U :  \mathcal{P}_p(M) \rightarrow \R$ is a viscosity \emph{sub-solution} of \eqref{WHJB} whenever for all $\mu \in \mathcal{P}_p(M)$ and $\gamma \in \partial^+ U(\mu)$, we have
\begin{align*}
	U(\mu) + \textbf{H}(\gamma) \leq \mathcal{F}(\mu).
\end{align*}
We say that a lower-semicontinuous function $U : \mathcal{P}_p(M) \rightarrow \R$ is a viscosity \emph{super-solution} of \eqref{WHJB} whenever for all $\mu \in \mathcal{P}_p(M)$ and $\gamma \in \partial^- U(\mu)$, we have
\begin{align*}
	U(\mu) + \textbf{H}(\gamma) \geq \mathcal{F}(\mu).
\end{align*} We say that $U$ is a viscosity \emph{solution} if it is both a sub- and super-solution.
\end{definition}

\subsection{Standard non-convex cases}
We state and prove the analogous of Theorem \ref{p: comparisongeneral} in the Wasserstein setting. We shall work under the following assumption.
\begin{assumption}\label{wassumptionHami}
The following regularity on the Hamiltonian $H: \mathcal{P}_p(M)\times T^*M \to \R$ extends the regularity \eqref{continuityHamiltonian} to the Wasserstein framework: there exists a constant $C>0$ s.t.
\begin{equation}\label{continuityHamiltonian}
   |{H}(\mu,x,J_{p}(x,v)) - {H}(\nu, y,J_{p}(y,w))| \leq C\big(1 + \|v\|^{p-1}_x + \|w\|^{p-1}_y \big)( D_S((x,v), (y,w)) + W_p(\mu,\nu))
\end{equation}
holds $ \forall (x,v),(y,w) \in TM.$

If either the manifold $M$ is not compact or $p\not=2$ we have to prescribe a further condition: there exists a uniform constant $C>0$
 \begin{equation*}
   |{H}(\mu,x,J_{p}(x,v)+J_p(x,w)) - {H}(\mu, x,J_{p}(x,w))| \leq C (1+ \|w\|_x+ \|v\|_{x})\|v\|^{p-1}_x \quad \forall (x,v),(x,w)\in TM.  
\end{equation*}
\end{assumption}

The main result of the section is the following.
\begin{theorem}\label{p: wcomparison general}
Let $H: \mathcal{P}_p(M)\times T^*M \to \R$ be satisfying Assumption \ref{wassumptionHami}. Let $F_{i}: \mathcal P_p(M) \to \R, i=0,1$ be two functions s.t. $\mathcal{F}_0$ is u.s.c., and $\mathcal {F}_1$ is l.s.c. .
Let $U_i: [0,T]\times \mathcal P_p( M )\to \R$, $i=0,1$ s.t.

\begin{itemize}
    \item $U_0$ has bounded super-level sets and is weakly upper semi continuous in $\mathcal P_p(M)$,
    \item $U_0$ is a sub-solution of $$W(\mu)+\int_M H(\mu,x,d_\mu W)d\mu(x)=\mathcal F_{0}(\mu), \quad  \mu \in M.$$
\end{itemize}
\begin{itemize}
    \item $U_1$ has bounded sub-level sets and is weakly lower semi continuous in $\mathcal P_p(M)$,
    \item $U_1$ is a super-solution of $$W(\mu)+\int_M H(\mu,x,d_\mu W)d\mu(x)=\mathcal F_{1}(\mu), \quad  \mu \in M.$$
\end{itemize}

Suppose in addition that $W^p_p(\cdot,\cdot)$ on $\mathcal{P}_p(M)\times \mathcal{P}_p(M)$ is super-differentiable.

Then,
$$\sup_{ \mathcal P_p(M)}U_0-U_1\leq \sup_{ \mathcal P_p(M)} \mathcal F_0-\mathcal F_{1}.$$ 
\end{theorem}

\begin{remark}
We insist upon the assumption of super-differentiability of $W_p^p$: we do not know any simple characterization of this fact but already gave in Remark \ref{rem:superwp} several examples of such situations.
\end{remark}

\begin{proof}
The proof closely follows that of Theorem \ref{p: comparisongeneral}. 
For $\varepsilon>0$, set 
\begin{equation*}
\Phi^{\varepsilon}(\mu,\nu):=U_0(\mu)-U_1(\nu)-\frac{W_p^p(\mu,\nu)}{p\varepsilon^{p-1}}
\end{equation*}
Let $(\mu_\varepsilon,\nu_\varepsilon) \subset \mathcal{P}_p(M)\times \mathcal{P}_p(M)$ be a sequence of maximum points for $\Phi^{\varepsilon,\delta}$. Such a sequence exists thanks to the weak u.s.c. in $\mathcal{P}_p(M)$ of $U_0 - U_1$ the fact that it has bounded super-level sets. Then, as in the proof of Theorem \ref{p: comparisongeneral}, one can show that the penalization $\frac{W^p(\mu_\varepsilon,\nu_\varepsilon)}{p\varepsilon^{p-1}}$ vanishes as $\varepsilon\to 0$.

For any  $\sigma^{+}_{\varepsilon}, \sigma^{-}_{\varepsilon}\in \mathcal{P}_p(TM)$ two optimal plans that send $\mu_{\varepsilon}$ into $\nu_{\varepsilon}$, and viceversa, respectively, by Proposition \ref{wsuperdifferentiability}
\begin{equation}
    \begin{cases}
        &(\pi_1\times -\varepsilon^{1-p}\pi_2)_{\#}\mathcal{J}_{p}(\sigma_{\varepsilon}^{+})\in \partial^{+}U_0(\mu_\varepsilon),\\
        &(\pi_1\times \varepsilon^{1-p}\pi_2)_{\#}\mathcal{J}_p(\sigma_{\varepsilon}^{-}) \in \partial^{-}U_1(\nu_\varepsilon).
    \end{cases}
\end{equation}

In the previous, we used the notation $(\pi_1\times \lambda\pi_2)(x,v) = (x,\lambda v)$ for $(x,v) \in TM, \lambda \in \R$. We want to show that we can choose the two couplings so that they concentrate on the same geodesics, but in opposite sense in time. This is true because of Theorem \ref{t: equivalent} which states the equivalence with the dynamical formulation. In particular, we consider a probability measure $\eta_\varepsilon$ on path in $C([0,\varepsilon],M)$ which is optimal for the dynamical formulation associated to $W_p^p$. We then simply choose $\sigma^+_\varepsilon := (e_0,e'_0)_\#\eta_\varepsilon$ and $\sigma^-_\varepsilon := (e_1,-e'_1)_\#\eta_\varepsilon$, where $e_t$ is the evaluation map at time $t$ and $e'_t$ is the evaluation of the derivative at time $t$. Moreover, since $\eta_{\varepsilon}$ is concentrated on minimizing geodesic we have, $\forall \varepsilon>0$
 $\|\dot\gamma(0)\|_{\gamma(0)}=\|\dot\gamma_{\varepsilon}(\varepsilon)\|_{\gamma({\varepsilon})}=\frac{d(\gamma(0),\gamma(\varepsilon))}{\varepsilon},\quad  \gamma-\eta_{\varepsilon} \,a.e.$  

The definition of viscosity sub and super solution yields
\begin{align*}
U_0(\mu_{\varepsilon}) - U_1(\nu_{\varepsilon}) 
&\leq \underbrace{-\textbf{H}((\pi_1,-\varepsilon^{1-p}\pi_2)_{\#}\mathcal{J}_{p}(\sigma^{+}_{\varepsilon}))+\textbf{H}((\pi_1\times \varepsilon^{1-p}\pi_2)_{\#}\mathcal{J}_p(\sigma_{\varepsilon}^{-}))}_{\Delta H}
   + \mathcal{F}_{0}(\mu_{\varepsilon}) - \mathcal{F}_{1}(\nu_{\varepsilon}). 
\end{align*}
We estimate the difference between the two Hamiltonians through
\begin{align*}
   \Delta H&= \int_{C([0,\varepsilon])}\Big(H(\mu_{\varepsilon},\gamma(0), J_{p}(\gamma(0), \dot{\gamma}(0))) - H(\nu_{\varepsilon},\gamma(\varepsilon), J_{p}(\gamma(\varepsilon), \dot{\gamma}(\varepsilon)))\Big) d\eta_{\varepsilon}(\gamma) \\
&\underbrace{\leq}_{Ass. \ref{wassumptionHami}} C\int_{C([0,\varepsilon])} \Big(1 +\|\dot \gamma(0)\|^{p-1}_{\gamma(0)}+\|\dot \gamma(\varepsilon)\|^{p-1}_{\gamma({\varepsilon})}\Big)\big(D_S((\gamma(0),\dot \gamma(0)),(\gamma(\varepsilon),\dot \gamma(\varepsilon))\big)+ W_p(\mu_\varepsilon, \nu_{\varepsilon})) d\eta_{\varepsilon}(\gamma)\\
&\underbrace{\leq}_{\eqref{PTgeodesic}, \eqref{e: isometrydual}} C\int_{C([0,\varepsilon])} \left(1 + \|\dot \gamma(0)\|^{p-1}_{\gamma(0)}+\|\dot \gamma({\varepsilon})\|_{\gamma(\varepsilon)}^{p-1}\right) \big(d(\gamma(0),\gamma(\varepsilon))+W_p(\mu_{\varepsilon},\nu_{\varepsilon})\big))d\eta_{\varepsilon}(\gamma)\\
&\leq C \left (W_p(\mu_{\varepsilon},\nu_{\varepsilon})+ 
\frac{W^{p}_p(\mu_{\varepsilon},\nu_{\varepsilon})}{\varepsilon^{p-1}}\right).
\end{align*}
Therefore we infer that 
$$
U(\mu_{\varepsilon})-U(\nu_{\varepsilon})\leq  \mathcal{F}_0(\mu_{\varepsilon}) - \mathcal{F}_{1}(\nu_{\varepsilon})  + C \big (W_p(\mu_{\varepsilon},\nu_{\varepsilon})+ 
\frac{W^{p}_p(\mu_{\varepsilon},\nu_{\varepsilon})}{\varepsilon^{p-1}}\big).
$$
The proof then concludes as in the finite dimensional.
\end{proof}
\begin{remark}
Just as in the finite dimensional case, the assumptions of the boundedness of sub/super level sets of $V/U$ can be addressed by subtracting/adding a penalization of the type 
\begin{equation*}
    \delta \mathcal{M}_{p}(\mu)=\delta W^p_{p}(\mu,\delta_{x_0}),
\end{equation*}
for a $\delta>0$. One can again compensate possible $p$-growth in the non-compact case, localizing the optimization procedure to bounded sets. The proof with this additional assumption follows the same superposition argument as in the compact case with the further localization term to bound in both the penalization and in the equation. It is at this point that we need to use the bound
 \begin{equation*}
   |{H}(\mu,x,J_{p}(x,v)+J_p(x,w)) - {H}(\mu, x,J_{p}(x,w))| \leq C (1+ \|w\|_x+ \|v\|_{x})\|v\|^{p-1}_x \quad \forall (x,v),(x,w)\in TM.  
 \end{equation*}
 that is not necessary if we already know a priori that we can restrict our attention to a bounded set. Let us insist upon the fact that the non-local compactness of $\mathcal P_p(M)$ remains an issue in this case, and that we still need to assume weak continuity of the sub/super solutions.
\end{remark}

\subsection{The case of a more geometric Hamiltonian}
We now present our main result, or rather proof, concerning functions that can be used in the doubling of variables techniques in $\mathcal P_p(M)$, when the Hamiltonian $H$ is geometric. Furthermore we state the following result in the compact case and refer to the previous proof for techniques to handle non-compact case. Note that this compactness assumption shall imply that the function used as a penalization is super-differentiable.

\begin{proposition}\label{p: wcomparison}Fix $p>1$. Suppose $M$ compact. 
Let $L:TM \to \R$ be a $C^2$, reversible and dissipative Lagrangian, and let $H=L^*$ its associated geometric Hamiltonian. Consider two functions $\mathcal F_0,\mathcal F_1:\mathcal P_p(M)\to\R$ such that: $\mathcal F_0$ is  u.s.c. and $\mathcal F_1$ is l.s.c..

Let $U_0,U_1:\mathcal P_p(M)\to\R$ be, respectively, a u.s.c bounded-above subsolution and a l.s.c. bounded-below supersolution of 
\begin{equation*}
    W(\mu)+\int_M H\big(x,d_\mu W\big)\,d\mu(x) \;=\; \mathcal F_i(\mu),
    \qquad \mu\in\mathcal P_p(M),  \end{equation*}
where $i =0,1.$

Then,
$$\sup_{\mathcal{P}_p(M)}U_0-U_1\leq \sup_{\mathcal{P}_p(M)} \mathcal F_0- \mathcal F_{1}.$$
\end{proposition}

\begin{proof}
Set
$$
    \Phi^{\varepsilon} (\mu,\nu):=U_0(\mu)-U_1(\nu)-D_M(\varepsilon,\mu,\nu).
$$
Let $(\mu_\varepsilon,\nu_{\varepsilon})_{\varepsilon>0} \subset \mathcal{P}_p(M)\times \mathcal{P}_p(M)$ be a sequence of maximum points for $\Phi^{\varepsilon}$. Such a sequence exists thanks to the boundeness, u.s.c. of $\Phi^{\varepsilon}$, and the compactness. 
 By Lemma \ref{WConvergence Maxim}, we deduce the existence of a subsequence  of $((\mu_\varepsilon,\nu_\varepsilon))_{\varepsilon}$ (that we do not relabel), converging to a maximum point of $(\bar \mu,\bar \mu)\in M$ of the difference $U_0-U_1$ over $\mathcal{P}_p(M)$.

Let  $\sigma^{+}_{\varepsilon}, \sigma^{-}_{\varepsilon}\in \mathcal{P}_p(TM)$ be two optimal plans for the cost $D(\varepsilon,\cdot,\cdot)$ that send $\mu_{\varepsilon}$ into $\nu_{\varepsilon}$, and viceversa, respectively.
Then, Proposition \ref{wsuperdifferentiability} yields
\begin{equation}
    \begin{cases}
        &(\pi_1, -\pi_2)_{\#}\mathcal{L}_{P}(\sigma^{+}_{\varepsilon})\in \partial^+U_0(\mu_{\varepsilon})\\
        &\mathcal{L}_{P}(\sigma_{\varepsilon}^{-})\in \partial^{-}U_1(\nu_\varepsilon).
\end{cases}
\end{equation}
Just as in the previous proof, we do not want to choose any such elements but rather consider $\eta_{\varepsilon}$ which is probability measure on $C([0,\varepsilon,M)$ which is optimal for the dynamic formulation of $D(\varepsilon,\mu_\varepsilon,\nu_\varepsilon)$. We then define $\sigma_\varepsilon^+ := (\mathcal L^{-1}\circ (\pi_1,\pi_2)_\#\eta_\varepsilon$ and $\sigma_{\varepsilon}^{-}$ analogously. Note that we have the following conservation of energy $\eta_\varepsilon$ almost everywhere.
\begin{align}\label{comparison: conservationofenergy}
H\left(\gamma(0), \frac{\partial}{\partial v} L(\gamma(0), \dot{\gamma}(0))\right) - H\left(\gamma(\varepsilon), \frac{\partial}{\partial v} L(\gamma(\varepsilon), \dot{\gamma}(\varepsilon))\right)
= \hat{H}(\gamma(0), \dot{\gamma}(0))) - \hat{H}(\gamma(0), \dot{\gamma}(\varepsilon)))\underbrace{=}_{\text{Theorem} \, \ref{t:EL}}0.
\end{align} 
Thus, by definition of viscosity sub and super solutions 
\begin{align*}
    U_0(\mu_{\varepsilon})-U_1(\nu_\varepsilon)&\leq -\textbf{H}((\pi,-\text{Id})_{\#}\mathcal{L}_{P}(\sigma^{+}_{\varepsilon}))+\textbf{H}(\mathcal{L}_{P}(\sigma^{-}_{\varepsilon}))+ \mathcal{F}_{0}(\mu_\varepsilon)- \mathcal F_{1}(\nu_\varepsilon)\\
    &= \int H(x, w^{+})d\sigma^{+}_{\varepsilon}(x,w^{+}) - \int H(y, w^{-})d\sigma^{-}_{\varepsilon}(x,w^{-}) + \mathcal F_{0}(\mu_\varepsilon)- \mathcal F_{1}(\nu_\varepsilon)\\
    &=\int \Big(H(\gamma(0), \frac{\partial}{\partial v} L(\gamma(0), \dot{\gamma}(0))) - H(\gamma(\varepsilon), \frac{\partial}{\partial v} L(\gamma(\varepsilon), \dot{\gamma}(\varepsilon)))\Big) d\eta_{\varepsilon}(\gamma)+ \mathcal F_{0}(\mu_\varepsilon)- \mathcal F_{1}(\nu_\varepsilon)\\
    &\underbrace{=}_{\eqref{comparison: conservationofenergy}}\mathcal F_{0}(\mu_\varepsilon)- \mathcal F_{1}(\nu_\varepsilon).
\end{align*}
Sending $\varepsilon \to 0$ and using the regularity properties of the functions, we get 
$$
\max_{\mathcal P_p(M)} U_0-U_1 \leq \mathcal F_{0}(\bar \mu)-\mathcal F_1(\bar \mu)\leq \max_{ \mathcal P_p(M)} \mathcal F_0 -\mathcal F_1.$$
\end{proof}

\section{Perspectives}
We conclude this paper by recalling the main motivation of our work and insisting upon what we believe are its main takeaways. The current development of a theory of HJB equations on spaces of measures hint that a global comprehension of comparison principles is still missing. We here argue that using the richness of the geometry of the spaces of probability measures can lead to quite easy comparison principles, whereas, arguing only with what would be smooth test functions does not seem to work.

By giving quite exact computation depending on the penalization used in the doubling of variables argument, we justify the use of what is the \emph{natural} distance on the underlying space. This fact will be of the upmost importance in a forthcoming work \cite{bertucciceccherini} on the study of HJB equations on sets of positive measures, in which the results above suggest to work with Wasserstein-Fisher-Rao or Hellinger-Kantorovich distances.

\section{Acknowledgments}
The first author acknowledge a partial support from the Lagrange Mathematics and Computing Research Center and the Chair FDD (Institut Louis Bachelier). The research of the first author is funded by the ERC Starting Grant PaDiESeM. 

The second author gratefully acknowledges the support provided by the ?Ing. Aldo Gini? Foundation and by the LYSM during the development of this work.

 \bibliographystyle{plain} 
\bibliography{BiblioNote}

\end{document}